\documentclass[11pt,a4paper]{amsart}
\usepackage{verbatim}
\usepackage{setspace}
\usepackage[left=2.5cm,right=2.5cm,top=3.0cm,bottom=3.0cm,includeheadfoot]{geometry}
\usepackage{latexsym}
\usepackage{enumitem}
\usepackage[font=small,labelfont=bf]{caption}
\usepackage{pgfplots}
\usepackage{float}

\usepackage[utf8]{inputenc}
\usepackage[T1]{fontenc}
\usepackage[english]{babel}
\usepackage{csquotes} 

\usepackage{amsmath, amsthm, amssymb}
\usepackage{ytableau}      
\usepackage{ragged2e}
\usepackage{multicol}
\usepackage[rgb]{xcolor}

\usepackage{graphicx, rotating, float}
\usepackage{tikz}
\usetikzlibrary{babel}
\usetikzlibrary{arrows, bending, positioning, chains}
\usetikzlibrary{automata, shapes.geometric}
\usetikzlibrary{calc, arrows.meta}

\tikzset{
	myedge/.style={->, -{Latex[#1]}}
}

\newcommand{\R}{\mathbb{R}}
\newcommand{\Z}{\mathbb{Z}}

\theoremstyle{plain}
\newtheorem{theorem}{Theorem}[section]
\newtheorem{lemma}[theorem]{Lemma}
\newtheorem{proposition}[theorem]{Proposition}
\newtheorem{corollary}[theorem]{Corollary}

\theoremstyle{definition}
\newtheorem{definition}[theorem]{Definition}
\newtheorem{example}[theorem]{Example}

\theoremstyle{remark}
\newtheorem{remark}[theorem]{Remark}

\usepackage{xcolor}
\definecolor{darkblue}{RGB}{0,0,139}

\usepackage[hidelinks]{hyperref}
\hypersetup{
	colorlinks=true,
	linkcolor=darkblue,
	citecolor=darkblue,
	urlcolor=darkblue,
	pdftitle={Weyl groups and the Modified Kostant Game},
	pdfauthor={Alexander Caviedes Castro, Juan Sebastián Cortés-Cruz}
}

\title{Weyl Groups and the Modified Kostant Game}

\author[Alexander Caviedes Castro]{Alexander Caviedes Castro}
\address{Department of Mathematics, Universidad Nacional de Colombia}
\email{acaviedesc@unal.edu.co}

\author[Juan Sebastián Cortés-Cruz]{Juan Sebastián Cortés-Cruz}
\address{Department of Mathematics, Universidad Nacional de Colombia}
\email{juascortescru@unal.edu.co}

\date{\today}

\begin{document}
	
	\maketitle
	
\begin{abstract}
This paper presents a generalization of the Kostant game, a combinatorial framework originally for generating positive roots in Lie algebras. By introducing an arbitrary multi-vertex modification, we prove that the resulting game configurations naturally biject with the minimal length representatives of parabolic quotients $W/W_J$. This yields a dynamical and algorithmic perspective on reduced words. Finally, we apply this framework to derive a novel root counting identity, formalize the Coxeter-theoretic foundation for combinatorial approaches to the Mukai conjecture, establish the regularity of reduced word languages via finite state automata, and dynamically construct Standard Young Tableaux.
\end{abstract}

\section{Introduction}

The classification of complex semisimple Lie algebras relies fundamentally on the combinatorial and geometric properties of root systems and their associated Weyl groups. While the static properties of these structures—encoded in Dynkin diagrams and Cartan matrices—are classical and thoroughly understood \cite{Humphreys1972}, exploring dynamical or algorithmic methods to generate and interact with these structures remains an active area of research in algebraic combinatorics. 

One such dynamical approach is the \textit{Kostant game}, a chip-firing process played on the nodes of a Dynkin diagram. Originally conceived as a combinatorial mechanism to generate the positive roots of a Lie algebra, the classical Kostant game highlights the deeply algorithmic nature of root systems. It provides a visual and discrete step-by-step method to understand how simple roots combine, culminating in the highest root.

In this paper, we establish a novel combinatorial framework at the intersection of Lie theory and algebraic combinatorics by introducing an arbitrary \textbf{multi-vertex generalization of the Kostant game}. Building upon the single-vertex modified game recently utilized to study symplectic flag varieties \cite{CaviedesCastro2022}, our Modified game allows for the simultaneous modification (or activation) of any subset of vertices in the Dynkin diagram. 

The central contribution of this work is the proof that the valid sequences of moves in this Modified game are in natural bijection with the elements of the quotient $W/W_J$, where $W$ is the Weyl group and $W_J$ is the parabolic subgroup determined by the unmodified vertices. We rigorously demonstrate that the game's valid firing sequences correspond precisely to the reduced expressions of the minimal length coset representatives, denoted as $W^J$.

This bijection is not merely of theoretical interest; it serves as a powerful computational and combinatorial tool. First, we demonstrate how this formalism provides the underlying Coxeter-theoretic foundation for geometric problems, specifically explaining the combinatorial algorithms used to compute roots of generalized Hilbert polynomials in the context of the Mukai conjecture \cite{CaviedesCastro2022}. Furthermore, we show its independent utility by proving a novel root counting identity, establishing the regularity of reduced word languages in $W^J$ via finite state automata, and dynamically constructing Standard Young Tableaux. By shifting the perspective from static algebra to dynamical systems, this Modified Kostant Game offers a tangible, algorithmic lens through which the complex geometry of flag varieties and Coxeter groups can be analyzed.

\section{Preliminaries: Root Systems and Parabolic Quotients}
\label{sec:prelim}

We briefly review the foundational concepts of root systems and Weyl groups, focusing on the structures strictly necessary for our results. We assume the reader is familiar with the standard theory of semisimple Lie algebras (see, e.g., \cite{Humphreys1972, BjornerBrenti2005}).


\subsection{Root Systems and Weyl Groups}

Let $E$ be a finite-dimensional Euclidean space endowed with a positive-definite inner product $\langle \cdot, \cdot\rangle$. Let $\Phi \subset E$ be a reduced crystallographic root system, and choose a base of simple roots $\Delta = \{\alpha_1, \dots, \alpha_n\}$. The set of positive roots is denoted by $\Phi^+$. 

For each root $\alpha \in \Phi$, the associated coroot is $\alpha^\vee = \frac{2\alpha}{\langle\alpha, \alpha\rangle}$. The \textit{Cartan matrix} $A = (a_{ij})$ of $\Phi$ is given by the integers $a_{ij} = \langle \alpha_i, \alpha_j^\vee \rangle$.

The combinatorial structure of $\Phi$ is elegantly encoded in its \textbf{Dynkin diagram}. This is a directed multigraph constructed as follows: we assign a vertex to each simple root $\alpha_i \in \Delta$. For any pair of distinct vertices $i$ and $j$, the number of edges connecting them is given by the product $a_{ij}a_{ji} \in \{0, 1, 2, 3\}$. When two vertices are connected by multiple edges, which occurs if and only if the corresponding simple roots have different lengths, an arrow is drawn pointing from the vertex associated with the longer root towards the one associated with the shorter root. The Dynkin diagram completely determines the Cartan matrix, and consequently, the isomorphism class of the root system and the Coxeter presentation of the Weyl group.

The \textbf{Weyl group} $W$ associated with $\Phi$ is the finite reflection group generated by the set of simple reflections $S = \{s_1, \dots, s_n\}$, where the action of $s_i$ on $E$ is defined by:
\begin{equation}
    s_i(\lambda) = \lambda - \langle \lambda, \alpha_i^\vee \rangle \alpha_i \quad \text{for } \lambda \in E.
\end{equation}
The pair $(W, S)$ forms a Coxeter system. The \textit{length function} $\ell(w)$ for an element $w \in W$ is defined as the minimum integer $k$ such that $w$ can be expressed as a product of $k$ simple reflections, $w = s_{i_1} s_{i_2} \dots s_{i_k}$. Such an expression is called a \textit{reduced word}. 

To connect this algebraic definition with the geometry of the root system, we define the \textit{inversion set} of $w$ as the set of positive roots that are sent to negative roots by the action of $w$:
\begin{equation}
    \mathcal{I}(w) := \Phi^+ \cap w^{-1}(-\Phi^+).
\end{equation}
It is a classical and fundamental result that the length of an element is exactly the cardinality of its inversion set \cite[(1.5)]{Humphreys1972}:
\begin{equation}
    \ell(w) = |\mathcal{I}(w)|.
\end{equation}


\subsection{Parabolic Subgroups and Quotients}

Let $J \subseteq S$ be a subset of the simple reflections. The subgroup $W_J$ generated by $J$ is called a \textbf{parabolic subgroup} of $W$. The pair $(W_J, J)$ is itself a Coxeter system.

In the study of flag varieties and algebraic combinatorics, the left cosets $W/W_J$ play a central role. Each coset $wW_J$ contains a unique element of minimal length, which serves as a canonical representative.

\begin{definition}[Minimal Coset Representatives]
	The set of minimal length representatives for the quotient $W/W_J$ is denoted by $W^J$ and is defined in terms of inversion sets as:
	\begin{equation}
		W^J := \{w \in W \mid \mathcal{I}(w) \subseteq \Phi^+ \setminus \Phi_J^+\}, 
	\end{equation}
	where $\Phi_J^+$ is the set of positive roots in the root system generated by the base $\{\alpha_j \mid j \in J\}$, that is, $ \Phi_J^+ = \Phi^+ \cap \text{Span}_{\mathbb{R}}(\{\alpha_j\}_{j \in J})$.
 
    Equivalently, $W^J$ is classically characterized by the right descent condition; it consists of those elements whose length strictly increases when multiplied on the right by any generator in $J$ \cite[Proposition~2.4.4]{BjornerBrenti2005}:
    \begin{equation}
        W^J = \{w \in W \mid \ell(ws_j) > \ell(w) \text{ for all } j \in J\}.
    \end{equation}
\end{definition}

The structure of $W$ factors nicely through these representatives. The following standard proposition will be crucial when establishing the bijection with the Modified Kostant Game.

\begin{proposition}[Parabolic Decomposition] \label{prop:parabolic_decomp}
	Every element $w \in W$ can be uniquely factored as $w = w^J \cdot w_J$, where $w^J \in W^J$ and $w_J \in W_J$. Furthermore, lengths are strictly additive under this decomposition:
	\begin{equation}
		\ell(w) = \ell(w^J) + \ell(w_J).
	\end{equation}
	(See \cite[Proposition~2.4.4]{BjornerBrenti2005}.)
\end{proposition}

The elements of $W^J$ encode the combinatorial "shape" of the quotient space. Our primary objective is to demonstrate that these abstract algebraic representatives can be generated algorithmically by applying a modified set of chip-firing rules to the Dynkin diagram of $\Phi$.


\section{The Kostant Game}
\label{sec:kostant_game}

A recurring problem when studying root systems is determining efficiently which linear combinations of simple roots actually form roots within the system. The Kostant game, popularized in combinatorial literature (e.g., \cite{chen2017}), provides a dynamic and constructive solution to this problem. It allows us to generate positive roots by playing a ``chip-firing'' game on the Dynkin diagram.


\subsection{Definition on Simple Graphs}

To build intuition, we first define the game on simple, undirected graphs, which correspond to simply-laced root systems (types $A_n, D_n, E_n$).

\begin{definition}[Kostant Game]
	\label{def:kostant_game}
	Let $\Gamma = (V, E)$ be a simple graph with vertex set $V = \{1, \dots, n\}$. Let $N(i)$ denote the set of neighbors of a vertex $i \in V$. A \textbf{configuration} is a vector $c = (c_1, \dots, c_n) \in \Z_{\ge 0}^n$. 
	
	A vertex $i$ is said to be \textbf{sad} (or unhappy) in configuration $c$ if its value is strictly less than half the sum of the values of its neighbors:
	\begin{equation}
		c_i < \frac{1}{2} \sum_{j \in N(i)} c_j.
	\end{equation}
    If a vertex is not sad, it is either \textbf{happy} ($c_i = \frac{1}{2} \sum_{j \in N(i)} c_j$) or \textbf{excited} ($c_i > \frac{1}{2} \sum_{j \in N(i)} c_j$).\\
	
	A \textbf{move} consists of choosing a sad vertex $i$ and ``firing'' it. The firing operation updates the configuration $c$ to a new configuration $c'$, where the value of $i$ is replaced by the sum of its neighbors minus its current value, while all other vertices remain unchanged:
	\begin{equation}
		c'_i = \sum_{j \in N(i)} c_j - c_i, \quad \text{and} \quad c'_k = c_k \text{ for } k \neq i.
	\end{equation}
\end{definition}

The game ends when no vertex is sad, and if it is possible to terminate the game on $\Gamma$ we say that the graph is \textbf{Kostant finite}. A fundamental result states that a graph is Kostant finite for all valid initial configurations if and only if $\Gamma$ is a simply-laced connected Dynkin diagram \cite{chen2017}.

\begin{remark}
    While the game terminates on a simply-laced Dynkin diagram for any non-negative initial configuration, its connection to Lie theory is most direct when starting from a standard basis vector $c = e_i$ (representing a simple root $\alpha_i$). In this case, the set of all configurations reachable during the game corresponds precisely to the set of positive roots $\Phi^+$ of the associated Lie algebra \cite{chen2017}. Moreover, regardless of the sequence of moves, the game always terminates in the same final configuration, which corresponds to the highest root of $\Phi$ \cite{chen2017}.
\end{remark}

\begin{example}
Below is an illustrative example of the game on a graph that is Kostant finite, specifically on $D_4$, reviewing the entire set of possible configurations on this graph with their respective transitions and final configuration.

\begin{figure}[H]
    \centering

    \tikzset{
        vtx/.style={circle, draw, fill=white, inner sep=1pt, minimum size=18pt, font=\small},
    }
    
    \newcommand{\DfourConfig}[4]{
        \begin{tikzpicture}[remember picture]
            \node[vtx] (v1) at (0,0) {$#1$};
            \node[vtx] (v2) at (1.2,0) {$#2$};
            \node[vtx] (v3) at (2.4,0) {$#3$};
            \node[vtx] (v4) at (1.2,-1.2) {$#4$};
            \draw (v1) -- (v2) -- (v3);
            \draw (v2) -- (v4);
        \end{tikzpicture}
    }
    
    \begin{tikzpicture}[scale=0.7, transform shape]
        \node (c1000) at (-9, 10) {\DfourConfig{1}{}{}{}};
        \node (c0100) at (-3, 10) {\DfourConfig{}{1}{}{}};
        \node (c0010) at (3, 10)  {\DfourConfig{}{}{1}{}};
        \node (c0001) at (9, 10)  {\DfourConfig{}{}{}{1}};

        \node (c1100) at (-9, 6) {\DfourConfig{1}{1}{}{}};
        \node (c0110) at (0, 6)  {\DfourConfig{}{1}{1}{}};
        \node (c0101) at (9, 6)  {\DfourConfig{}{1}{}{1}};

        \node (c1110) at (-6, 2) {\DfourConfig{1}{1}{1}{}};
        \node (c1101) at (0, 2)  {\DfourConfig{1}{1}{}{1}};
        \node (c0111) at (6, 2)  {\DfourConfig{}{1}{1}{1}};
        
        \node (c1111) at (0, -2) {\DfourConfig{1}{1}{1}{1}};

        \node (c1211) at (0, -6) {\DfourConfig{1}{2}{1}{1}};

        \draw[->, thick] (c1000) -- (c1100);
        \draw[->, thick] (c0100) -- (c1100);
        
        \draw[->, thick] (c0100) -- (c0110);
        \draw[->, thick] (c0010) -- (c0110);
        
        \draw[->, thick] (c0100) -- (c0101);
        \draw[->, thick] (c0001) -- (c0101);

        \draw[->, thick] (c1100) -- (c1110);
        \draw[->, thick] (c0110) -- (c1110);

        \draw[->, thick] (c1100) -- (c1101);
        \draw[->, thick] (c0101) -- (c1101);
        
        \draw[->, thick] (c0110) -- (c0111);
        \draw[->, thick] (c0101) -- (c0111);

        \draw[->, thick] (c1110) -- (c1111);
        \draw[->, thick] (c1101) -- (c1111);
        \draw[->, thick] (c0111) -- (c1111);
        
        \draw[->, thick] (c1111) -- (c1211);

    \end{tikzpicture}
    \caption{The set of possible configurations of the Kostant game on $D_4$.}
    \label{fig:kostant_d4}
\end{figure}
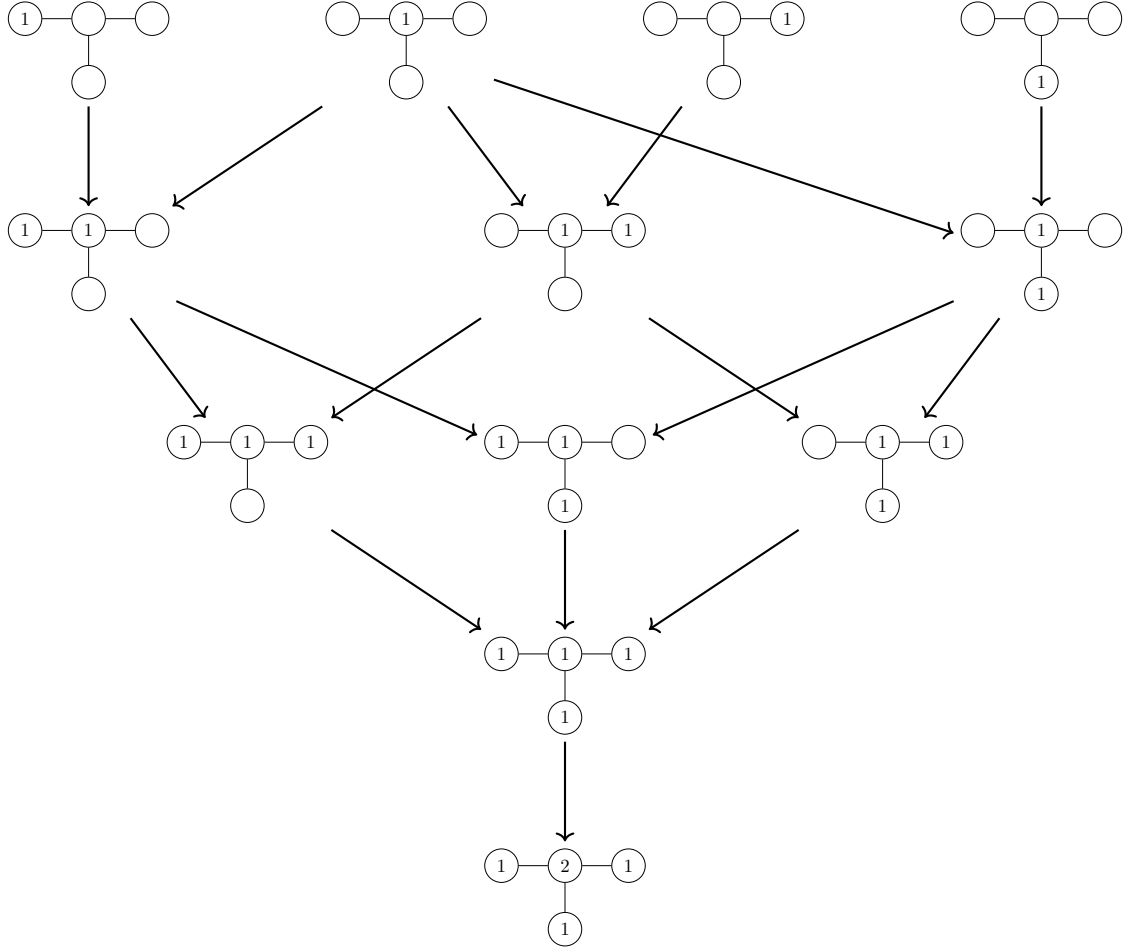
\end{example}

\begin{example}[Infinite game on the affine graph $\tilde{D}_4$]
	To contrast with the finite behavior on classical Dynkin diagrams, we illustrate the divergence of the game on the affine graph $\tilde{D}_4$. The graph $\tilde{D}_4$ consists of a central vertex $v_0$ connected to four peripheral vertices $v_1, v_2, v_3$, and $v_4$.
	
	\begin{figure}[H]
		\centering
		\scalebox{0.7}{
			\begin{tikzpicture}
				\node[circle, draw, thick, minimum size=20pt, font=\Large] (v0) at (0,0) {$v_0$};
				\node[circle, draw, thick, minimum size=20pt, font=\Large] (v1) at (0:2.5) {$v_1$};
				\node[circle, draw, thick, minimum size=20pt, font=\Large] (v2) at (90:2.5) {$v_2$};
				\node[circle, draw, thick, minimum size=20pt, font=\Large] (v3) at (180:2.5) {$v_3$};
				\node[circle, draw, thick, minimum size=20pt, font=\Large] (v4) at (270:2.5) {$v_4$};
				\draw[thick] (v0) -- (v1) (v0) -- (v2) (v0) -- (v3) (v0) -- (v4);
			\end{tikzpicture}
		}
		\caption{The structure of the affine graph $\tilde{D}_4$.}
		\label{fig:affine_d4}
	\end{figure}
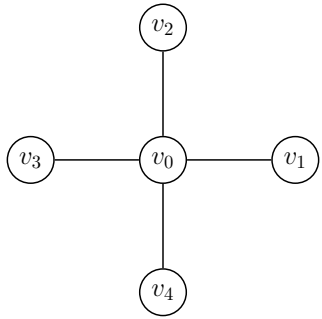
	
	Let $c = (c_0, c_1, c_2, c_3, c_4)$ denote the chip configuration. Start the game with the configuration $c_A = (1, 1, 1, 1, 1)$. We observe how this configuration evolves, demonstrating that the game cannot terminate.
	
	\begin{figure}[H]
		\centering
		\begin{tikzpicture}
			\begin{scope}[xshift=-5cm, scale=0.5]
				\node[circle, draw, thick, minimum size=10pt, font=\large] (v0) at (0,0) {1};
				\node[circle, draw, thick, minimum size=10pt, font=\large] (v1) at (0:2.5) {1};
				\node[circle, draw, thick, minimum size=10pt, font=\large] (v2) at (90:2.5) {1};
				\node[circle, draw, thick, minimum size=10pt, font=\large] (v3) at (180:2.5) {1};
				\node[circle, draw, thick, minimum size=10pt, font=\large] (v4) at (270:2.5) {1};
				\draw[thick] (v0) -- (v1) (v0) -- (v2) (v0) -- (v3) (v0) -- (v4);
				\node[text width=5cm, text centered, below=0.5cm] at (0,-3.5) {\textbf{Configuration $c_A$:} \\ $c=(1,1,1,1,1)$. \\ Only $v_0$ is sad since $1 < \frac{1}{2}(4) = 2$.};
			\end{scope}
			
			\node at (0, 0) {\huge$\longrightarrow$};
			
			\begin{scope}[xshift=5cm, scale=0.5]
				\node[circle, draw, thick, minimum size=10pt, font=\large] (v0) at (0,0) {3};
				\node[circle, draw, thick, minimum size=10pt, font=\large] (v1) at (0:2.5) {1};
				\node[circle, draw, thick, minimum size=10pt, font=\large] (v2) at (90:2.5) {1};
				\node[circle, draw, thick, minimum size=10pt, font=\large] (v3) at (180:2.5) {1};
				\node[circle, draw, thick, minimum size=10pt, font=\large] (v4) at (270:2.5) {1};
				\draw[thick] (v0) -- (v1) (v0) -- (v2) (v0) -- (v3) (v0) -- (v4);
				\node[text width=5cm, text centered, below=0.5cm] at (0,-3.5) {\textbf{Configuration $c_B$:} \\ $v_0 \mapsto -1 + 4 = 3$. \\ Now all peripheral vertices are sad.};
			\end{scope}
		\end{tikzpicture}
		\caption{A diverging transition of the Kostant game on $\tilde{D}_4$.}
		\label{fig:evolution_d4_affine}
	\end{figure}
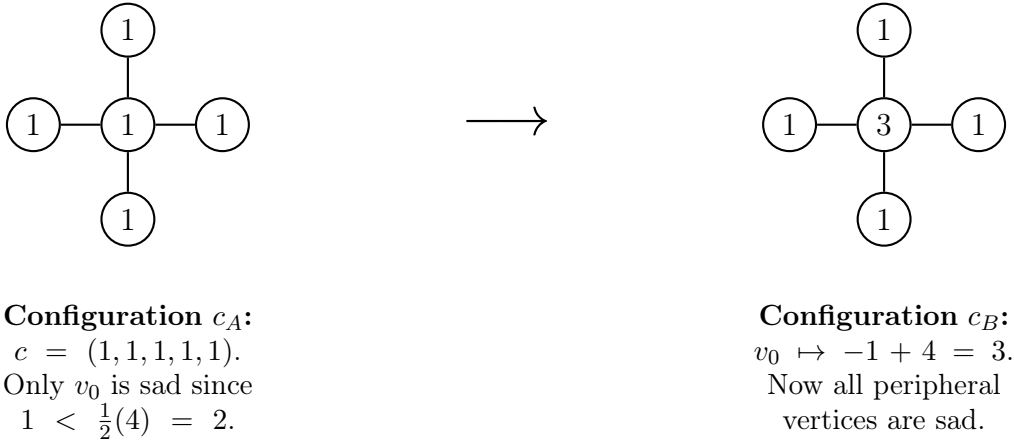
	
	In configuration $c_B = (3, 1, 1, 1, 1)$, the dynamics unfold as follows:
	\begin{itemize}
		\item \textbf{Central vertex $v_0$:} Its value is $3$. The sum of its neighbors is $1+1+1+1=4$. Since $3 \not< 4/2 = 2$, $v_0$ is no longer sad.
		\item \textbf{Peripheral vertices $v_1, \dots, v_4$:} Each has a value of $1$. Their only neighbor is $v_0$, which has a value of $3$. Since $1 < 3/2 = 1.5$, all four peripheral vertices are strictly sad. Firing any peripheral vertex updates its value to $-1 + 3 = 2$. 
		\item \textbf{Infinite cycle of sadness:} Once all peripheral vertices are fired, they will each hold a value of $2$. The configuration becomes $(3, 2, 2, 2, 2)$. At this point, the central vertex $v_0$ evaluates its sadness condition again: the sum of its neighbors is $2 \times 4 = 8$. Since $3 < 8/2 = 4$, $v_0$ becomes sad once more, updating its value to $-3 + 8 = 5$.
	\end{itemize}
	
	This establishes a cyclic instability. Firing the peripheral vertices strictly increases their values and forces the central vertex to become sad again. When the central vertex fires, it strictly increases its own value and immediately induces sadness in all peripheral vertices. Consequently, the game admits an infinite sequence of valid moves, and the chip configuration grows unbounded. This confirms that the game on $\tilde{D}_4$ will never end.
\end{example}

\subsection{Multiply-Laced Diagrams}

The rules of the Kostant game naturally extend to multiply-laced Dynkin diagrams (types $B_n$, $C_n$, $F_4$, $G_2$). Here, edges are directed and may be multiple, so the adjacency must be weighted according to the root lengths.

For two adjacent vertices $i$ and $j$, denote by $n_{i,j}$ the number of arrows pointing from $j$ to $i$ in the diagram. Our convention is:
\begin{itemize}
    \item If $\alpha_i$ and $\alpha_j$ have the same length, then $n_{i,j}=1=n_{j,i}$.
    \item If $\alpha_i$ is long and $\alpha_j$ is short, then $n_{i,j}=1$ and $n_{j,i}=2$, or $n_{i,j}=1$ and $n_{j,i}=3$ (the only possibilities in the crystallographic cases $F_4$ and $G_2$).
\end{itemize}
Thus $n_{i,j}$ is simply the multiplicity of the edge directed from $j$ to $i$.

A configuration is a vector $c = (c_1,\dots,c_n)\in \Z_{\ge 0}^n$, just as in the simply‑laced case. For each vertex $i$, the state is determined by comparing $c_i$ with half the weighted sum of its neighbors:
\begin{itemize}
    \item $i$ is \textbf{happy} if $\displaystyle c_i = \frac12 \sum_{j\in N(i)} n_{i,j}\,c_j$,
    \item $i$ is \textbf{sad} if $\displaystyle c_i < \frac12 \sum_{j\in N(i)} n_{i,j}\,c_j$,
    \item $i$ is \textbf{excited} if $\displaystyle c_i > \frac12 \sum_{j\in N(i)} n_{i,j}\,c_j$.
\end{itemize}
A \textbf{move} consists of choosing a sad vertex $i$ and firing it, which updates the configuration to $c'$ by
\[
c'_i = -c_i + \sum_{j\in N(i)} n_{i,j}\,c_j, \qquad c'_k = c_k\;\text{for }k\neq i.
\]

When the game starts from a standard basis vector $c=e_i$ (representing a simple root $\alpha_i$), the set of all configurations reachable during the game corresponds exactly to the set of positive roots $\Phi^+$ of the associated Lie algebra~\cite{chen2017}. In contrast with the simply‑laced case, the final configuration need not be unique; the asymmetry introduced by the multiple edges can force the game into different terminals depending on the initial moves.

\begin{example}[Bifurcation in $F_4$]
	When playing the Kostant game on the $F_4$ diagram, the presence of the double arrow between the short and long roots breaks the perfect symmetry of the game. Depending on whether the initial moves are concentrated before the short root or after the long root, the game branches into two distinct valid paths, terminating in two different final configurations. The asymmetry introduced by the multiple edges restricts the state space to one of two possible sinks.

    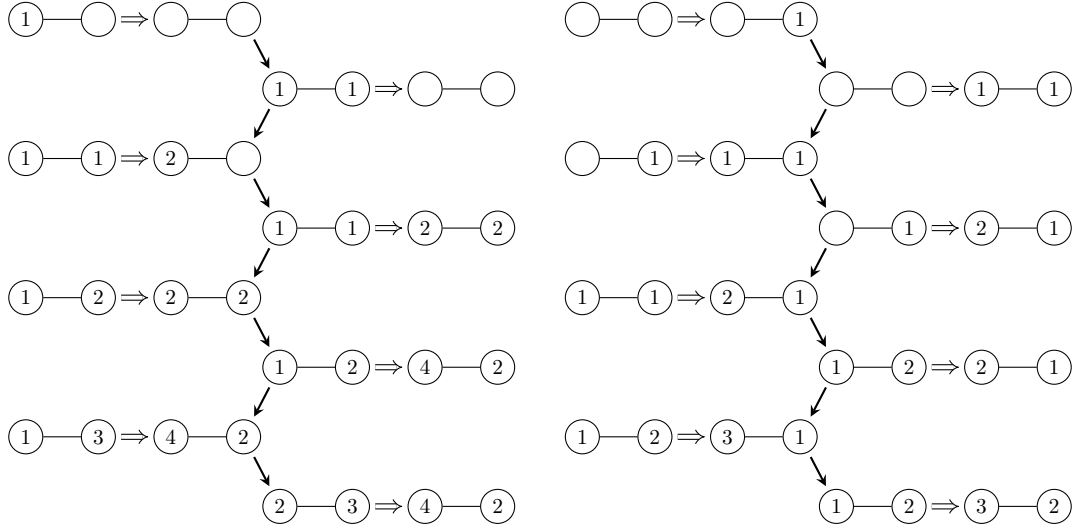
\begin{figure}[H]
    \centering
    \begin{tikzpicture}[scale=0.8, transform shape]
        
        \tikzset{
            double-arrow/.style={double, double distance=1.5pt, -implies, shorten >=2pt, shorten <=2pt},
            vtx/.style={circle, draw, fill=white, inner sep=1pt, minimum size=16pt, font=\small},
            connect/.style={->, >=stealth, thick, shorten >=2pt, shorten <=2pt}
        }
        
        \begin{scope}[shift={(0,0)}]
            \node[vtx] (L1_1) at (0,0) {1};
            \node[vtx] (L1_2) at (1.2,0) {};
            \node[vtx] (L1_3) at (2.4,0) {};
            \node[vtx] (L1_4) at (3.6,0) {};
            \draw (L1_1) -- (L1_2);
            \draw[double-arrow] (L1_2) -- (L1_3);
            \draw (L1_3) -- (L1_4);
        \end{scope}
        
        \begin{scope}[shift={(4.2, -1.15)}]
            \node[vtx] (L2_1) at (0,0) {1};
            \node[vtx] (L2_2) at (1.2,0) {1};
            \node[vtx] (L2_3) at (2.4,0) {};
            \node[vtx] (L2_4) at (3.6,0) {};
            \draw (L2_1) -- (L2_2);
            \draw[double-arrow] (L2_2) -- (L2_3);
            \draw (L2_3) -- (L2_4);
        \end{scope}
        \draw[connect] (L1_4) -- (L2_1);
        
        \begin{scope}[shift={(0, -2.3)}]
            \node[vtx] (L3_1) at (0,0) {1};
            \node[vtx] (L3_2) at (1.2,0) {1};
            \node[vtx] (L3_3) at (2.4,0) {2};
            \node[vtx] (L3_4) at (3.6,0) {};
            \draw (L3_1) -- (L3_2);
            \draw[double-arrow] (L3_2) -- (L3_3);
            \draw (L3_3) -- (L3_4);
        \end{scope}
        \draw[connect] (L2_1) -- (L3_4);
        
        \begin{scope}[shift={(4.2, -3.45)}]
            \node[vtx] (L4_1) at (0,0) {1};
            \node[vtx] (L4_2) at (1.2,0) {1};
            \node[vtx] (L4_3) at (2.4,0) {2};
            \node[vtx] (L4_4) at (3.6,0) {2};
            \draw (L4_1) -- (L4_2);
            \draw[double-arrow] (L4_2) -- (L4_3);
            \draw (L4_3) -- (L4_4);
        \end{scope}
        \draw[connect] (L3_4) -- (L4_1);
        
        \begin{scope}[shift={(0, -4.6)}]
            \node[vtx] (L5_1) at (0,0) {1};
            \node[vtx] (L5_2) at (1.2,0) {2};
            \node[vtx] (L5_3) at (2.4,0) {2};
            \node[vtx] (L5_4) at (3.6,0) {2};
            \draw (L5_1) -- (L5_2);
            \draw[double-arrow] (L5_2) -- (L5_3);
            \draw (L5_3) -- (L5_4);
        \end{scope}
        \draw[connect] (L4_1) -- (L5_4);
        
        \begin{scope}[shift={(4.2, -5.75)}]
            \node[vtx] (L6_1) at (0,0) {1};
            \node[vtx] (L6_2) at (1.2,0) {2};
            \node[vtx] (L6_3) at (2.4,0) {4};
            \node[vtx] (L6_4) at (3.6,0) {2};
            \draw (L6_1) -- (L6_2);
            \draw[double-arrow] (L6_2) -- (L6_3);
            \draw (L6_3) -- (L6_4);
        \end{scope}
        \draw[connect] (L5_4) -- (L6_1);
        
        \begin{scope}[shift={(0, -6.9)}]
            \node[vtx] (L7_1) at (0,0) {1};
            \node[vtx] (L7_2) at (1.2,0) {3};
            \node[vtx] (L7_3) at (2.4,0) {4};
            \node[vtx] (L7_4) at (3.6,0) {2};
            \draw (L7_1) -- (L7_2);
            \draw[double-arrow] (L7_2) -- (L7_3);
            \draw (L7_3) -- (L7_4);
        \end{scope}
        \draw[connect] (L6_1) -- (L7_4);
        
        \begin{scope}[shift={(4.2, -8.05)}]
            \node[vtx] (L8_1) at (0,0) {2};
            \node[vtx] (L8_2) at (1.2,0) {3};
            \node[vtx] (L8_3) at (2.4,0) {4};
            \node[vtx] (L8_4) at (3.6,0) {2};
            \draw (L8_1) -- (L8_2);
            \draw[double-arrow] (L8_2) -- (L8_3);
            \draw (L8_3) -- (L8_4);
        \end{scope}
        \draw[connect] (L7_4) -- (L8_1);
        
        \begin{scope}[shift={(9.2, 0)}]
            \node[vtx] (R1_1) at (0,0) {};
            \node[vtx] (R1_2) at (1.2,0) {};
            \node[vtx] (R1_3) at (2.4,0) {};
            \node[vtx] (R1_4) at (3.6,0) {1};
            \draw (R1_1) -- (R1_2);
            \draw[double-arrow] (R1_2) -- (R1_3);
            \draw (R1_3) -- (R1_4);
        \end{scope}
        
        \begin{scope}[shift={(13.4, -1.15)}]
            \node[vtx] (R2_1) at (0,0) {};
            \node[vtx] (R2_2) at (1.2,0) {};
            \node[vtx] (R2_3) at (2.4,0) {1};
            \node[vtx] (R2_4) at (3.6,0) {1};
            \draw (R2_1) -- (R2_2);
            \draw[double-arrow] (R2_2) -- (R2_3);
            \draw (R2_3) -- (R2_4);
        \end{scope}
        \draw[connect] (R1_4) -- (R2_1);
        
        \begin{scope}[shift={(9.2, -2.3)}]
            \node[vtx] (R3_1) at (0,0) {};
            \node[vtx] (R3_2) at (1.2,0) {1};
            \node[vtx] (R3_3) at (2.4,0) {1};
            \node[vtx] (R3_4) at (3.6,0) {1};
            \draw (R3_1) -- (R3_2);
            \draw[double-arrow] (R3_2) -- (R3_3);
            \draw (R3_3) -- (R3_4);
        \end{scope}
        \draw[connect] (R2_1) -- (R3_4);
        
        \begin{scope}[shift={(13.4, -3.45)}]
            \node[vtx] (R4_1) at (0,0) {};
            \node[vtx] (R4_2) at (1.2,0) {1};
            \node[vtx] (R4_3) at (2.4,0) {2};
            \node[vtx] (R4_4) at (3.6,0) {1};
            \draw (R4_1) -- (R4_2);
            \draw[double-arrow] (R4_2) -- (R4_3);
            \draw (R4_3) -- (R4_4);
        \end{scope}
        \draw[connect] (R3_4) -- (R4_1);
        
        \begin{scope}[shift={(9.2, -4.6)}]
            \node[vtx] (R5_1) at (0,0) {1};
            \node[vtx] (R5_2) at (1.2,0) {1};
            \node[vtx] (R5_3) at (2.4,0) {2};
            \node[vtx] (R5_4) at (3.6,0) {1};
            \draw (R5_1) -- (R5_2);
            \draw[double-arrow] (R5_2) -- (R5_3);
            \draw (R5_3) -- (R5_4);
        \end{scope}
        \draw[connect] (R4_1) -- (R5_4);
        
        \begin{scope}[shift={(13.4, -5.75)}]
            \node[vtx] (R6_1) at (0,0) {1};
            \node[vtx] (R6_2) at (1.2,0) {2};
            \node[vtx] (R6_3) at (2.4,0) {2};
            \node[vtx] (R6_4) at (3.6,0) {1};
            \draw (R6_1) -- (R6_2);
            \draw[double-arrow] (R6_2) -- (R6_3);
            \draw (R6_3) -- (R6_4);
        \end{scope}
        \draw[connect] (R5_4) -- (R6_1);
        
        \begin{scope}[shift={(9.2, -6.9)}]
            \node[vtx] (R7_1) at (0,0) {1};
            \node[vtx] (R7_2) at (1.2,0) {2};
            \node[vtx] (R7_3) at (2.4,0) {3};
            \node[vtx] (R7_4) at (3.6,0) {1};
            \draw (R7_1) -- (R7_2);
            \draw[double-arrow] (R7_2) -- (R7_3);
            \draw (R7_3) -- (R7_4);
        \end{scope}
        \draw[connect] (R6_1) -- (R7_4);
        
        \begin{scope}[shift={(13.4, -8.05)}]
            \node[vtx] (R8_1) at (0,0) {1};
            \node[vtx] (R8_2) at (1.2,0) {2};
            \node[vtx] (R8_3) at (2.4,0) {3};
            \node[vtx] (R8_4) at (3.6,0) {2};
            \draw (R8_1) -- (R8_2);
            \draw[double-arrow] (R8_2) -- (R8_3);
            \draw (R8_3) -- (R8_4);
        \end{scope}
        \draw[connect] (R7_4) -- (R8_1);
        
    \end{tikzpicture}
    \caption{Two paths in the Kostant Game on the $F_4$ diagram ending in different configurations.}
    \label{fig:kostant_f4_final}
\end{figure}
\end{example}


\subsection{Definition with External Sources}

Let $\Phi$ be a root system with simple roots $\Delta = \{\alpha_1, \dots, \alpha_n\}$ and Cartan matrix $A$. We select a subset of vertices $I \subseteq \{1, \dots, n\}$ to act as \textbf{modified nodes} or ``external sources''. The corresponding parabolic subgroup will be determined by the unmodified vertices $J = S \setminus I$.

\begin{definition}[Modified Kostant Game with active set $I$]
    A \textbf{configuration} is a vector $c = (c_1,\dots,c_n) \in \mathbb{Z}_{\ge 0}^n$.
    
    For a vertex $v$, let $N(v)$ denote its neighbors in the Dynkin diagram, and let $n_{v,u}$ be the number of arrows pointing from $u$ to $v$. The state of a vertex is defined as follows:
    \begin{itemize}
        \item $v$ is \textbf{sad} (or unhappy) if
        \begin{equation}
            c_v < \frac12 \Biggl( \sum_{u \in N(v)} n_{v,u}\,c_u \;+\; \sum_{p \in I} \delta_{v,p} \Biggr),
        \end{equation}
        \item $v$ is \textbf{happy} if equality holds,
        \item $v$ is \textbf{excited} if the inequality is reversed ($>$).
    \end{itemize}
    If a vertex $v$ is sad, we may \textbf{fire} it, updating the configuration to $c'$ given by
    \begin{equation}
        c'_v = \sum_{u \in N(v)} n_{v,u}\,c_u + \sum_{p \in I} \delta_{v,p} - c_v,
    \end{equation}
    and $c'_k = c_k$ for all $k \neq v$. The game always starts from the zero configuration $c^{(0)} = \mathbf{0}$.
\end{definition}

The additional term $\sum_{p\in I}\delta_{v,p}$ acts as an external forcing: without it the zero configuration would be stable because every vertex would be (trivially) happy. The chosen set $I$ therefore plays the role of an activator that triggers the dynamics. In the sequel, we refer to this process as the \textbf{Modified Kostant Game on $I$} (or with active set $I$).

Note that, by definition, a vertex is happy or excited precisely when it is not sad, exactly as in the classical game.

\begin{example}
    Figure~\ref{fig:a4_mod2vertices} shows two instances of the modified game on $A_4$, one with $I=\{1\}$ and the other with $I=\{2\}$. In both cases, being simply‑laced diagrams, the set of reachable configurations converges to a unique final configuration, just as in the original Kostant game.
    
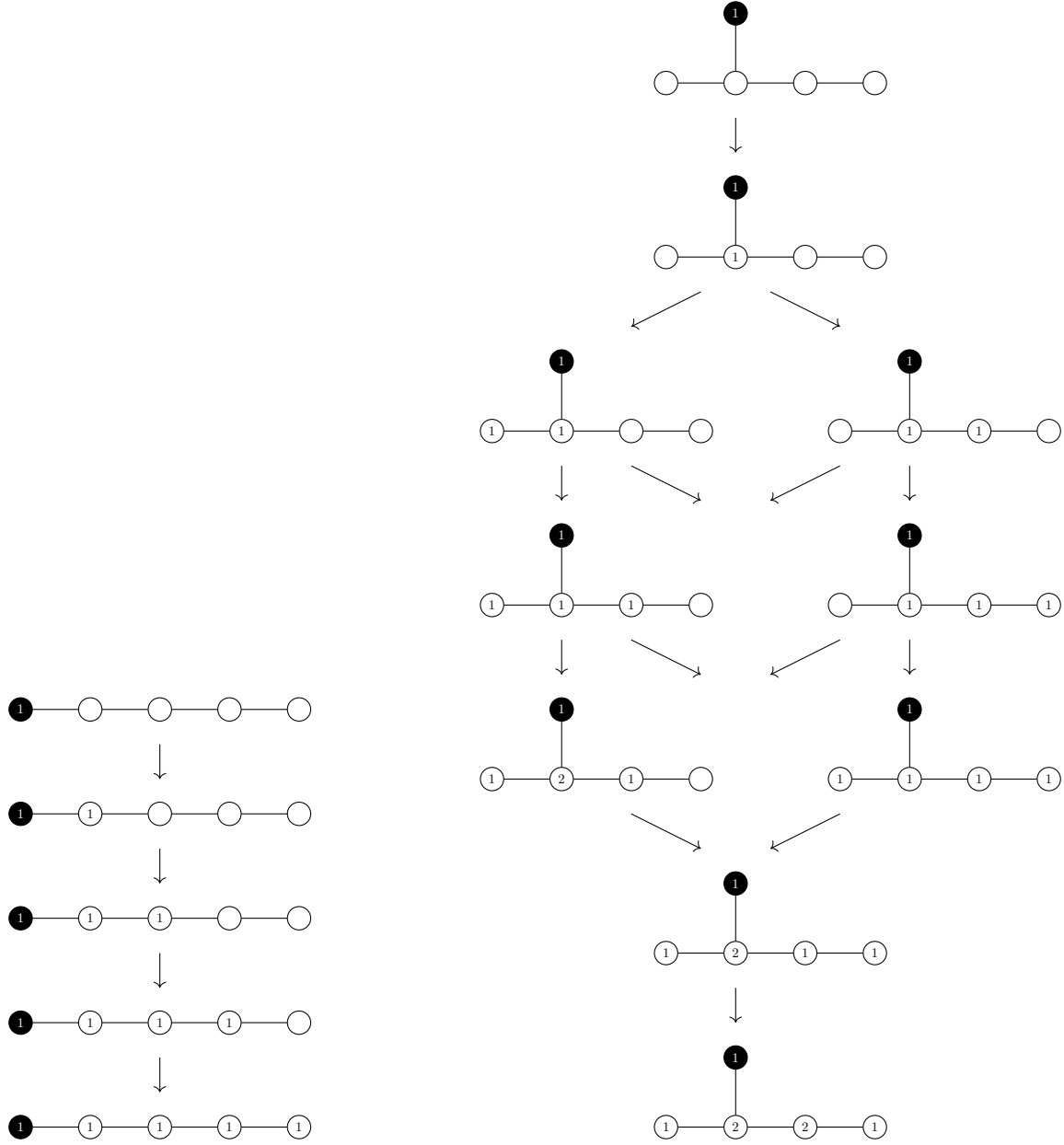
\begin{figure}[H]
\centering

\begin{tikzpicture}[scale=1]
  \node[draw,fill=black,shape=circle,scale=0.5] (a) at (0,0) {$\color{white}{1}$};
  \node[draw,fill=white,shape=circle,scale=0.5] (b) at (1,0) {\phantom{$0$}};
  \node[draw,fill=white,shape=circle,scale=0.5] (c) at (2,0) {\phantom{$0$}};
  \node[draw,fill=white,shape=circle,scale=0.5] (d) at (3,0) {\phantom{$0$}};
  \node[draw,fill=white,shape=circle,scale=0.5] (e) at (4,0) {\phantom{$0$}};
  \draw (a)--(b)--(c)--(d)--(e);
  \draw [->] (2,-0.5) -- (2,-1);
  \node[draw,fill=black,shape=circle,scale=0.5] (a) at (0,-1.5) {$\color{white}{1}$};
  \node[draw,fill=white,shape=circle,scale=0.5] (b) at (1,-1.5) {$1$};
  \node[draw,fill=white,shape=circle,scale=0.5] (c) at (2,-1.5) {\phantom{$0$}};
  \node[draw,fill=white,shape=circle,scale=0.5] (d) at (3,-1.5) {\phantom{$0$}};
  \node[draw,fill=white,shape=circle,scale=0.5] (e) at (4,-1.5) {\phantom{$0$}};
  \draw (a)--(b)--(c)--(d)--(e);
  \draw [->] (2,-2) -- (2,-2.5);
  \node[draw,fill=black,shape=circle,scale=0.5] (a) at (0,-3)
  {$\color{white}{1}$};
  \node[draw,fill=white,shape=circle,scale=0.5] (b) at (1,-3) {$1$};
  \node[draw,fill=white,shape=circle,scale=0.5] (c) at (2,-3) {$1$};
  \node[draw,fill=white,shape=circle,scale=0.5] (d) at (3,-3) {\phantom{$0$}};
  \node[draw,fill=white,shape=circle,scale=0.5] (e) at (4,-3) {\phantom{$0$}};
  \draw (a)--(b)--(c)--(d)--(e);
  \draw [->] (2,-3.5) -- (2,-4);
  \node[draw,fill=black,shape=circle,scale=0.5] (a) at (0,-4.5)
  {$\color{white}{1}$};
  \node[draw,fill=white,shape=circle,scale=0.5] (b) at (1,-4.5) {$1$};
  \node[draw,fill=white,shape=circle,scale=0.5] (c) at (2,-4.5) {$1$};
  \node[draw,fill=white,shape=circle,scale=0.5] (d) at (3,-4.5) {$1$};
  \node[draw,fill=white,shape=circle,scale=0.5] (e) at (4,-4.5) {\phantom{$0$}};
  \draw (a)--(b)--(c)--(d)--(e);
  \draw [->] (2,-5) -- (2,-5.5);
  \node[draw,fill=black,shape=circle,scale=0.5] (a) at (0,-6)
  {$\color{white}{1}$};
  \node[draw,fill=white,shape=circle,scale=0.5] (b) at (1,-6) {$1$};
  \node[draw,fill=white,shape=circle,scale=0.5] (c) at (2,-6) {$1$};
  \node[draw,fill=white,shape=circle,scale=0.5] (d) at (3,-6) {$1$};
  \node[draw,fill=white,shape=circle,scale=0.5] (e) at (4,-6) {$1$};
  \draw (a)--(b)--(c)--(d)--(e);
\end{tikzpicture}
\qquad \qquad \qquad
\begin{tikzpicture}[scale=1]
  \node[draw,shape=circle,scale=0.5] (a) at (-0.5,0) {\phantom{$0$}};
  \node[draw,fill=white,shape=circle,scale=0.5] (b) at (0.5,0) {\phantom{$0$}};
  \node[draw,fill=white,shape=circle,scale=0.5] (c) at (1.5,0) {\phantom{$0$}};
  \node[draw,fill=black,shape=circle,scale=0.5] (d) at (-0.5,1) {$\color{white}{1}$};
  \node[draw,fill=white,shape=circle,scale=0.5] (e) at (-1.5,0) {\phantom{$0$}};
  \draw (e) -- (a) -- (b)--(c);
  \draw (a) -- (d);
  \draw [->] (-0.5,-0.5) -- (-0.5,-1);
\node[draw,shape=circle,scale=0.5] (a) at (-0.5,-2.5) {$1$};
  \node[draw,fill=white,shape=circle,scale=0.5] (b) at (0.5,-2.5) {\phantom{$0$}};
  \node[draw,fill=white,shape=circle,scale=0.5] (c) at (1.5,-2.5) {\phantom{$0$}};
  \node[draw,fill=black,shape=circle,scale=0.5] (d) at (-0.5,-1.5) {$\color{white}{1}$};
  \node[draw,fill=white,shape=circle,scale=0.5] (e) at (-1.5,-2.5) {\phantom{$0$}};
  \draw (e) -- (a) -- (b)--(c);
  \draw (a) -- (d);
  \draw [->] (-1,-3) -- (-2,-3.5);
  \draw [->] (0,-3) -- (1,-3.5);
\node[draw,shape=circle,scale=0.5] (a) at (-3,-5) {$1$};
  \node[draw,fill=white,shape=circle,scale=0.5] (b) at (-2,-5) {\phantom{$0$}};
  \node[draw,fill=white,shape=circle,scale=0.5] (c) at (-1,-5) {\phantom{$0$}};
  \node[draw,fill=black,shape=circle,scale=0.5] (d) at (-3,-4) {$\color{white}{1}$};
  \node[draw,fill=white,shape=circle,scale=0.5] (e) at (-4,-5) {$1$};
  \draw (e) -- (a) -- (b)--(c);
  \draw (a) -- (d);
\node[draw,shape=circle,scale=0.5] (a) at (2,-5) {$1$};
  \node[draw,fill=white,shape=circle,scale=0.5] (b) at (3,-5) {$1$};
  \node[draw,fill=white,shape=circle,scale=0.5] (c) at (4,-5) {\phantom{$0$}};
  \node[draw,fill=black,shape=circle,scale=0.5] (d) at (2,-4) {$\color{white}{1}$};
  \node[draw,fill=white,shape=circle,scale=0.5] (e) at (1,-5) {\phantom{$0$}};
  \draw (e) -- (a) -- (b)--(c);
  \draw (a) -- (d);
  \draw [->] (-2,-5.5) -- (-1,-6);
  \draw [->] (1,-5.5) -- (0,-6);
  \draw [->] (-3,-5.5) -- (-3,-6);
  \draw [->] (2,-5.5) -- (2,-6);
  \node[draw,shape=circle,scale=0.5] (a) at (-3,-7.5) {$1$};
  \node[draw,fill=white,shape=circle,scale=0.5] (b) at (-2,-7.5) {$1$};
  \node[draw,fill=white,shape=circle,scale=0.5] (c) at (-1,-7.5) {\phantom{$0$}};
  \node[draw,fill=black,shape=circle,scale=0.5] (d) at (-3,-6.5) {$\color{white}{1}$};
  \node[draw,fill=white,shape=circle,scale=0.5] (e) at (-4,-7.5){$1$};
  \draw (e) -- (a) -- (b)--(c);
  \draw (a) -- (d);
  \node[draw,shape=circle,scale=0.5] (a) at (2,-7.5) {$1$};
  \node[draw,fill=white,shape=circle,scale=0.5] (b) at (3,-7.5) {$1$};
  \node[draw,fill=white,shape=circle,scale=0.5] (c) at (4,-7.5) {$1$};
  \node[draw,fill=black,shape=circle,scale=0.5] (d) at (2,-6.5) {$\color{white}{1}$};
  \node[draw,fill=white,shape=circle,scale=0.5] (e) at (1,-7.5){\phantom{$0$}};
  \draw (e) -- (a) -- (b)--(c);
  \draw (a) -- (d);
  \draw [->] (-2,-8) -- (-1,-8.5);
  \draw [->] (1,-8) -- (0,-8.5);
  \draw [->] (-3,-8) -- (-3,-8.5);
  \draw [->] (2,-8) -- (2,-8.5);
  \node[draw,shape=circle,scale=0.5] (a) at (-3,-10) {$2$};
  \node[draw,fill=white,shape=circle,scale=0.5] (b) at (-2,-10) {$1$};
  \node[draw,fill=white,shape=circle,scale=0.5] (c) at (-1,-10) {\phantom{$0$}};
  \node[draw,fill=black,shape=circle,scale=0.5] (d) at (-3,-9) {$\color{white}{1}$};
  \node[draw,fill=white,shape=circle,scale=0.5] (e) at (-4,-10){$1$};
  \draw (e) -- (a) -- (b)--(c);
  \draw (a) -- (d);
  \node[draw,shape=circle,scale=0.5] (a) at (2,-10) {$1$};
  \node[draw,fill=white,shape=circle,scale=0.5] (b) at (3,-10) {$1$};
  \node[draw,fill=white,shape=circle,scale=0.5] (c) at (4,-10) {$1$};
  \node[draw,fill=black,shape=circle,scale=0.5] (d) at (2,-9) {$\color{white}{1}$};
  \node[draw,fill=white,shape=circle,scale=0.5] (e) at (1,-10){$1$};
  \draw (e) -- (a) -- (b)--(c);
  \draw (a) -- (d);
   \draw [->] (-2,-10.5) -- (-1,-11);
  \draw [->] (1,-10.5) -- (0,-11);
  \node[draw,shape=circle,scale=0.5] (a) at (-0.5,-12.5) {$2$};
  \node[draw,fill=white,shape=circle,scale=0.5] (b) at (0.5,-12.5) {$1$};
  \node[draw,fill=white,shape=circle,scale=0.5] (c) at (1.5,-12.5) {$1$};
  \node[draw,fill=black,shape=circle,scale=0.5] (d) at (-0.5,-11.5) {$\color{white}{1}$};
  \node[draw,fill=white,shape=circle,scale=0.5] (e) at (-1.5,-12.5) {$1$};
  \draw (e) -- (a) -- (b)--(c);
  \draw (a) -- (d);
  \draw [->] (-0.5,-13) -- (-0.5,-13.5);
  \node[draw,shape=circle,scale=0.5] (a) at (-0.5,-15) {$2$};
  \node[draw,fill=white,shape=circle,scale=0.5] (b) at (0.5,-15) {$2$};
  \node[draw,fill=white,shape=circle,scale=0.5] (c) at (1.5,-15) {$1$};
  \node[draw,fill=black,shape=circle,scale=0.5] (d) at (-0.5,-14) {$\color{white}{1}$};
  \node[draw,fill=white,shape=circle,scale=0.5] (e) at (-1.5,-15) {$1$};
  \draw (e) -- (a) -- (b)--(c);
  \draw (a) -- (d);
\end{tikzpicture}
\caption{Two instances of the Modified Kostant Game on $A_4$: on the left with active set $I=\{1\}$, on the right with $I=\{2\}$.}
\label{fig:a4_mod2vertices}
\end{figure}
\end{example}
\subsection{Bijection with Minimal Coset Representatives}

The central result of our work is that this modified game completely models the structure of parabolic quotients. Let $J = S \setminus I$ be the set of unmodified vertices. 

\begin{theorem}
	\label{thm:bijection}
	There exists a canonical bijection between the set of all reachable configurations in the Modified Kostant Game (with active set $I$) and the set $W^J$ of minimal length coset representatives of $W/W_J$. Specifically, a sequence of valid vertex firings $(k_1, k_2, \dots, k_m)$ corresponds uniquely to a reduced expression $w = s_{k_m} \dots s_{k_1} \in W^J$.
\end{theorem}

To rigorously prove Theorem \ref{thm:bijection}, we first construct an algebraic model that parallels the combinatorial game. Let $E$ be the Euclidean space spanned by the simple roots. We introduce an extended vector space $E' = E \oplus \bigoplus_{p \in I} \R \beta_p$, where $\beta_p$ are formal basis vectors associated with the external sources. We extend the invariant bilinear form $\langle \cdot, \cdot \rangle$ to $E'$ such that $\langle \beta_p, \beta_q \rangle = \delta_{p,q}$ and $\langle \beta_p, \alpha_i^\vee \rangle = -\delta_{p,i}$ for all $p, q \in I$ and $i \in \{1, \dots, n\}$. 

The action of the Weyl group $W$ extends to $E'$ via the standard reflection formula 
\[
s_i(x) = x - \langle x, \alpha_i^\vee \rangle \alpha_i .
\]
For the source vectors one obtains $s_i(\beta_p) = \beta_p + \delta_{pi}\alpha_i$; one readily checks that the braid relations of $W$ are preserved, so this is a well‑defined action. 

Let $\beta = \sum_{p \in I} \beta_p \in E'$ be the total source vector. For a sequence of moves $(k_1, \dots, k_m)$, we define the Weyl group elements $w_m = s_{k_m} \dots s_{k_1}$ and the state vectors $v_m = w_m(\beta)$. 

By induction on $m$ one verifies that the game configuration $c^{(m)}$ is exactly recovered by $c^{(m)} = v_m - \beta$. Indeed, the reflection formula gives $s_v(\beta) = \beta + \sum_{p\in I}\delta_{pv}\,\alpha_v$, and a straightforward induction using the update rule of the game shows that the coefficients of the $\alpha_j$ in $v_m$ are precisely the chip numbers. Crucially, evaluating the coefficient of $\alpha_v$ in $s_v(v_m) - v_m$ reveals that the sadness condition at vertex $v$ is algebraically equivalent to $\langle v_m, \alpha_v^\vee \rangle < 0$. We now proceed to the proof of the bijection.

\begin{proof}[Proof of Theorem \ref{thm:bijection}]
    The proof relies on the algebraic equivalence established above. We first show that any valid sequence of moves yields a reduced expression of an element in $W^J$, and conversely that every reduced expression of an element of $W^J$ can be realised as a valid sequence of moves.

    \medskip
    \noindent\textit{From the game to the Weyl group.}
    Suppose that $(i_1, \dots, i_t)$ is a sequence of valid moves. Validity at step $\lambda$ (for $1 \le \lambda \le t$) is equivalent to the algebraic condition:
    \[
    \langle v_{\lambda-1}, \alpha_{i_\lambda}^\vee \rangle < 0.
    \]
    We define the integer quantity $K_\lambda$ from this condition:
    \[
    K_\lambda := -\langle v_{\lambda-1}, \alpha_{i_\lambda}^\vee \rangle.
    \]
    The validity of the move means that $K_\lambda$ must be a strictly positive integer. Using the $W$-invariance of the bilinear form, we can express $K_\lambda$ in terms of the initial state $\beta$:
    \[
    K_\lambda = -\langle w_{\lambda-1}(\beta), \alpha_{i_\lambda}^\vee \rangle = -\langle \beta, w_{\lambda-1}^{-1}(\alpha_{i_\lambda}^\vee) \rangle.
    \]

    The key object in this expression is the transformed co-root $\tilde{\gamma}^\vee := w_{\lambda-1}^{-1}(\alpha_{i_\lambda}^\vee)$. We will show that this co-root must be positive.

    The fundamental reason is that, by construction, the aggregate vector $\beta$ has a non-positive inner product with any positive co-root. By definition, $\langle \beta, \alpha_j^\vee \rangle = -\sum_{p\in I}\delta_{pj} \le 0$ for every simple root $\alpha_j$. Since every positive co-root $\gamma^\vee \in (\Phi^\vee)^+$ is a linear combination with non-negative integer coefficients of simple co-roots, it follows by linearity that $\langle \beta, \gamma^\vee \rangle \le 0$.

    Now, suppose for the sake of contradiction that $\tilde{\gamma}^\vee$ is a negative co-root. Then, its opposite $-\tilde{\gamma}^\vee$ would be a positive co-root. Calculating $K_\lambda$, we would have:
    \[
    K_\lambda = -\langle \beta, \tilde{\gamma}^\vee \rangle = \langle \beta, -\tilde{\gamma}^\vee \rangle \le 0.
    \]
    Since $-\tilde{\gamma}^\vee$ is a positive co-root, we know that this inner product must be $\le 0$. This implies that $K_\lambda \le 0$, which directly contradicts the condition that the move is valid (which requires $K_\lambda > 0$).

    Therefore, the transformed co-root $w_{\lambda-1}^{-1}(\alpha_{i_\lambda}^\vee)$ must be positive.

    Now, to find the explicit value of $K_\lambda$, we expand the positive co-root $w_{\lambda-1}^{-1}(\alpha_{i_\lambda}^\vee)$ in the basis of simple co-roots:
    \[
    w_{\lambda-1}^{-1}(\alpha_{i_\lambda}^\vee) = \sum_{j=1}^n k_j \alpha_j^\vee, \quad \text{where } k_j \in \mathbb{Z}_{\ge 0}.
    \]
    Substituting this expansion into our formula for $K_\lambda$:
    \begin{align*}
        K_\lambda &= -\left\langle \sum_{p \in I} \beta_p, \sum_{j=1}^n k_j \alpha_j^\vee \right\rangle = -\sum_{j=1}^n k_j \left\langle \sum_{p \in I} \beta_p, \alpha_j^\vee \right\rangle \\
        &= -\sum_{j=1}^n k_j \left( -\sum_{p \in I} \delta_{pj} \right) = \sum_{j \in I} k_j.
    \end{align*}
    The condition that the move is valid ($K_\lambda > 0$) translates to $\sum_{j \in I} k_j > 0$. Since all coefficients $k_j$ are non-negative integers, this implies that there must exist at least one index $j_0 \in I$ such that $k_{j_0} > 0$. This demonstrates that the co-root $w_{\lambda-1}^{-1}(\alpha_{i_\lambda}^\vee)$ cannot be written as a linear combination (with non-negative coefficients) of simple co-roots from $J=S\setminus I$. In other words:
    \[
    w_{\lambda-1}^{-1}(\alpha_{i_\lambda}^\vee) \in (\Phi^\vee)^+ \setminus (\Phi_J^\vee)^+.
    \]
    The map sending a root to its co-root ($\alpha \mapsto \alpha^\vee$) is a bijection from the set of roots $\Phi$ to $\Phi^\vee$ that is $W$-equivariant and preserves positivity. Therefore, the condition on the co-root is equivalent to a condition on the root $\tilde{\alpha}_\lambda := w_{\lambda-1}^{-1}(\alpha_{i_\lambda})$:
    \[
    \tilde{\alpha}_\lambda \in \Phi^+ \setminus \Phi_J^+.
    \]
    The first part, $\tilde{\alpha}_\lambda \in \Phi^+$, is the classical condition for the length of the Weyl element to increase: $\ell(s_{i_\lambda}w_{\lambda-1}) = \ell(w_{\lambda-1}) + 1$. Since this holds for all $\lambda$, the word $w_t = s_{i_t}\cdots s_{i_1}$ is a reduced expression.

    The second part, $\tilde{\alpha}_\lambda \notin \Phi_J^+$, relates to the definition of $W^J$. Observe that, given the definition $w_t = s_{i_t}\cdots s_{i_1}$, the inversion set of $w_t$ corresponds precisely to the set of generated roots:
    \[
    \mathcal{I}(w_t) = \{\tilde{\alpha}_1, \dots, \tilde{\alpha}_t\}.
    \]
    We have shown that all these roots are in $\Phi^+ \setminus \Phi_J^+$. Therefore, $\mathcal{I}(w_t) \subset \Phi^+ \setminus \Phi_J^+$, which by definition implies that $w_t \in W^J$.

    \medskip
    \noindent\textit{From the Weyl group to the game.}
    Conversely, let $w \in W^J$ and let $w = s_{i_t}\cdots s_{i_1}$ be one of its reduced expressions. For each $\lambda \in \{1, \dots, t\}$, we define $w_{\lambda-1} = s_{i_{\lambda-1}}\cdots s_{i_1}$.

    The fact that the expression is reduced implies that the root $\tilde{\alpha}_\lambda := w_{\lambda-1}^{-1}(\alpha_{i_\lambda})$ is positive for all $\lambda$. Furthermore, since $w \in W^J$, its inverse inversion set, $\mathcal{I}(w^{-1}) = \{\tilde{\alpha}_1, \dots, \tilde{\alpha}_t\}$, is contained in $\Phi^+ \setminus \Phi_J^+$. Therefore, for each $\lambda$, we have:
    \[
    \tilde{\alpha}_\lambda \in \Phi^+ \setminus \Phi_J^+.
    \]
    This implies that the corresponding co-root, $w_{\lambda-1}^{-1}(\alpha_{i_\lambda}^\vee)$, is in $(\Phi^\vee)^+ \setminus (\Phi_J^\vee)^+$. Expanding this co-root in the basis of simple co-roots:
    \[
    w_{\lambda-1}^{-1}(\alpha_{i_\lambda}^\vee) = \sum_{j=1}^n k_j \alpha_j^\vee, \quad \text{with } k_j \in \mathbb{Z}_{\ge 0},
    \]
    the condition of non-membership in $(\Phi_J^\vee)^+$ assures us that at least one coefficient $k_{j_0}$ with $j_0 \in I$ must be strictly positive.

    Now, we verify the validity of the move by calculating the integer $K_\lambda$:
    \[
    K_\lambda := -\langle v_{\lambda-1}, \alpha_{i_\lambda}^\vee \rangle = -\langle \beta, w_{\lambda-1}^{-1}(\alpha_{i_\lambda}^\vee) \rangle.
    \]
    Using the expansion of the co-root we just established, we find the same expression for $K_\lambda$ as in the previous implication:
    \[
    K_\lambda = \sum_{j \in I} k_j.
    \]
    Since we know that all $k_j$ are non-negative and at least one of them with index in $I$ is strictly positive, the sum $K_\lambda = \sum_{j \in I} k_j$ is a strictly positive integer.

    Given that $K_\lambda > 0$, the validity condition $\langle v_{\lambda-1}, \alpha_{i_\lambda}^\vee \rangle = -K_\lambda < 0$ holds. This inequality confirms that the move at vertex $i_\lambda$ is valid. Since this argument applies for all $\lambda$ from $1$ to $t$, the sequence of reflections $(i_1, \dots, i_t)$ corresponds to a valid sequence of moves in the modified game.

    \medskip
    \noindent\textit{Bijection between configurations and $W^J$.}
    The previous paragraphs show that a reachable configuration after $m$ moves corresponds uniquely to the element $w_m \in W^J$, and that every element of $W^J$ is realised by some reduced word, hence by some reachable configuration. It remains to check that two different sequences of moves giving the same configuration necessarily correspond to the same reduced expression (and therefore to the same element of $W^J$). 

    Observe first that the Modified Kostant Game possesses a unique final configuration. Indeed, the quotient $W/W_J$ contains a unique longest element $w_0^J$. Any maximal sequence of valid moves must yield a reduced expression for $w_0^J$, and the corresponding final state is $w_0^J(\beta)$. Consequently the final configuration is unique.

    Now suppose that $(k_1,\dots,k_m)$ and $(k_1',\dots,k_{m'}')$ are two sequences of valid moves leading to the same configuration $c$. Extend both sequences arbitrarily to maximal sequences (by adding moves until no sad vertex remains). Because the game always terminates, we obtain two sequences
    \[
    (k_1,\dots,k_m,l_1,\dots,l_p),\qquad (k_1',\dots,k_{m'}',l_1,\dots,l_p)
    \]
    that lead to the unique final configuration. Both extended sequences correspond to reduced words for $w_0^J$. Hence the Weyl group elements they represent are equal:
    \[
    s_{l_p}\cdots s_{l_1}s_{k_m}\cdots s_{k_1} = s_{l_p}\cdots s_{l_1}s_{k_{m'}'}\cdots s_{k_1'}.
    \]
    Canceling the common left factor $s_{l_p}\cdots s_{l_1}$, we obtain
    \[
    s_{k_m}\cdots s_{k_1} = s_{k_{m'}'}\cdots s_{k_1'}.
    \]
    Thus the two sequences correspond to the same element of $W^J$, which shows that the map from reachable configurations to $W^J$ is injective. Together with the surjectivity already established, this map is a bijection.
\end{proof}

\begin{example}[The Game on $A_2$ with both vertices modified]
	Consider the Dynkin diagram of type $A_2$. The Weyl group is the symmetric group $S_3$, generated by the simple reflections $s_1, s_2$ with the braid relation $s_1s_2s_1 = s_2s_1s_2$.
	
	We modify both vertices, i.e. $I = \{1,2\}$, so $J = \emptyset$ and $W^J = W(A_2)$. The aggregate source vector is $\beta = \beta_1 + \beta_2$, and the Cartan matrix of $A_2$ is 
	\[
	A = \begin{pmatrix} 2 & -1 \\ -1 & 2 \end{pmatrix}.
	\]
	
	The algebraic simulation proceeds as follows (we record the state $v_m = w_m(\beta)$ and the corresponding chip configuration $c^{(m)} = v_m - \beta$):
	\begin{itemize}
		\item \textbf{Step 0:} $w_0 = \mathrm{id}$, $v_0 = \beta$, $c^{(0)} = (0,0)$. 
		$\langle v_0, \alpha_1^\vee \rangle = -1 < 0$, so vertex $1$ is sad.
		\item \textbf{Step 1:} Fire vertex $1$. $w_1 = s_1$, $v_1 = \beta + \alpha_1$, $c^{(1)} = (1,0)$.
		\item \textbf{Step 2:} $\langle v_1, \alpha_2^\vee \rangle = -2 < 0$; fire vertex $2$. $w_2 = s_2s_1$, $v_2 = \alpha_1 + 2\alpha_2 + \beta$, $c^{(2)} = (1,2)$.
		\item \textbf{Step 3:} $\langle v_2, \alpha_1^\vee \rangle = -1 < 0$; fire vertex $1$. $w_3 = s_1s_2s_1$, $v_3 = 2\alpha_1 + 2\alpha_2 + \beta$, $c^{(3)} = (2,2)$. The game ends.
	\end{itemize}
	
	The word $w = s_1s_2s_1$ is a reduced expression of the longest element of $W(A_2)$. Its inversion set is $\mathcal{I}(w) = \{\alpha_1, \alpha_1+\alpha_2, \alpha_2\}$, the three positive roots of $A_2$. The configurations visited during the game are exactly the states that appear in the algebraic construction.
	
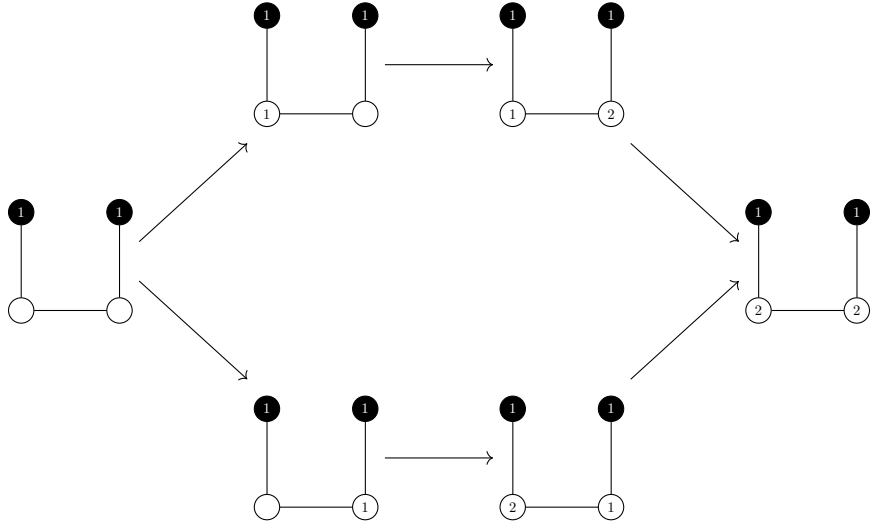
\begin{figure}[H]
\centering
\begin{tikzpicture}[scale=1.3]

  \node[draw,fill=black,shape=circle,scale=0.5] (a) at (0,0.5) {$\color{white}{1}$};
  \node[draw,fill=white,shape=circle,scale=0.5] (b) at (0,-0.5) {\phantom{$0$}};
  \node[draw,fill=white,shape=circle,scale=0.5] (c) at (1,-0.5) {\phantom{$0$}};
  \node[draw,fill=black,shape=circle,scale=0.5] (d) at (1,0.5) {$\color{white}{1}$};
  \draw (a) -- (b) -- (c) -- (d);

  \draw [->] (1.2, 0.2) -- (2.3, 1.2);
  \draw [->] (1.2, -0.2) -- (2.3, -1.2);

  \node[draw,fill=black,shape=circle,scale=0.5] (a) at (2.5,2.5) {$\color{white}{1}$};
  \node[draw,fill=white,shape=circle,scale=0.5] (b) at (2.5,1.5) {$1$};
  \node[draw,fill=white,shape=circle,scale=0.5] (c) at (3.5,1.5) {\phantom{$0$}};
  \node[draw,fill=black,shape=circle,scale=0.5] (d) at (3.5,2.5) {$\color{white}{1}$};
  \draw (a) -- (b) -- (c) -- (d);

  \draw [->] (3.7, 2.0) -- (4.8, 2.0);

  \node[draw,fill=black,shape=circle,scale=0.5] (a) at (2.5,-1.5) {$\color{white}{1}$};
  \node[draw,fill=white,shape=circle,scale=0.5] (b) at (2.5,-2.5) {\phantom{$0$}};
  \node[draw,fill=white,shape=circle,scale=0.5] (c) at (3.5,-2.5) {$1$};
  \node[draw,fill=black,shape=circle,scale=0.5] (d) at (3.5,-1.5) {$\color{white}{1}$};
  \draw (a) -- (b) -- (c) -- (d);

  \draw [->] (3.7, -2.0) -- (4.8, -2.0);

  \node[draw,fill=black,shape=circle,scale=0.5] (a) at (5,2.5) {$\color{white}{1}$};
  \node[draw,fill=white,shape=circle,scale=0.5] (b) at (5,1.5) {$1$};
  \node[draw,fill=white,shape=circle,scale=0.5] (c) at (6,1.5) {$2$};
  \node[draw,fill=black,shape=circle,scale=0.5] (d) at (6,2.5) {$\color{white}{1}$};
  \draw (a) -- (b) -- (c) -- (d);

  \draw [->] (6.2, 1.2) -- (7.3, 0.2);

  \node[draw,fill=black,shape=circle,scale=0.5] (a) at (5,-1.5) {$\color{white}{1}$};
  \node[draw,fill=white,shape=circle,scale=0.5] (b) at (5,-2.5) {$2$};
  \node[draw,fill=white,shape=circle,scale=0.5] (c) at (6,-2.5) {$1$};
  \node[draw,fill=black,shape=circle,scale=0.5] (d) at (6,-1.5) {$\color{white}{1}$};
  \draw (a) -- (b) -- (c) -- (d);

  \draw [->] (6.2, -1.2) -- (7.3, -0.2);

  \node[draw,fill=black,shape=circle,scale=0.5] (a) at (7.5,0.5) {$\color{white}{1}$};
  \node[draw,fill=white,shape=circle,scale=0.5] (b) at (7.5,-0.5) {$2$};
  \node[draw,fill=white,shape=circle,scale=0.5] (c) at (8.5,-0.5) {$2$};
  \node[draw,fill=black,shape=circle,scale=0.5] (d) at (8.5,0.5) {$\color{white}{1}$};
  \draw (a) -- (b) -- (c) -- (d);

\end{tikzpicture}
\caption{The Modified Kostant Game on $A_2$ altering two vertices.}
\label{fig:A2mod2ver}
\end{figure}
\end{example}


\subsection{Root Counting Identity}

In previous work \cite{CaviedesCastro2022}, it was shown that for a single active vertex ($I=\{j\}$), the total number of chips in the final configuration of the Modified Kostant Game equals the sum of the $j$-th coefficients of the positive roots outside the corresponding maximal parabolic subgroup. Using the bijection established in Theorem \ref{thm:bijection}, we now generalize this result to an arbitrary set of modified vertices $I \subseteq S$, extending the identity to any parabolic quotient $W/W_J$.

Let $\mathbf{c} = (c_1, \dots, c_n)$ be a configuration in the game. We define the total number of chips on the diagram as the $L^1$-norm of the configuration vector: $|\mathbf{c}| := \sum_{v=1}^n c_v$. 

For any coroot $\gamma^\vee \in \Phi^\vee$ with simple expansion $\gamma^\vee = \sum_{j=1}^n k_j \alpha_j^\vee$, we define its \textbf{$I$-height} as the sum of its coefficients corresponding to the modified set $I$:
\[
\mathrm{ht}_I(\gamma^\vee) := \sum_{j \in I} k_j .
\]

\begin{theorem}[Root Counting Identity]\label{thm:root_counting_identity}
	Let $I \subseteq S$ be an arbitrary subset of modified vertices, and $J = S \setminus I$. The Modified Kostant Game terminates in a unique final configuration $\mathbf{c}_{final}^{(I)}$. Moreover, the total number of chips in this final state equals the sum of the $I$-heights of all positive coroots that are not in the subsystem generated by $J$:
	\begin{equation}
		|\mathbf{c}_{final}^{(I)}| = \sum_{\gamma^\vee \in (\Phi^\vee)^+ \setminus (\Phi_J^\vee)^+} \mathrm{ht}_I(\gamma^\vee).
	\end{equation}
\end{theorem}

\begin{proof}
    By Theorem \ref{thm:bijection}, the game is finite and any maximal valid sequence of moves $(i_1, \dots, i_m)$ corresponds to a reduced expression for the longest element $w_0^J \in W^J$. Thus, the game terminates uniquely.
    
    Let us analyze the change in the total number of chips during a single step $\lambda$. At step $\lambda$, the vertex $i_\lambda$ is fired. The new chip count at this vertex is given by the game rule:
    \[
    c_{i_\lambda}^{(\lambda)} = -c_{i_\lambda}^{(\lambda-1)} + \sum_{u \in N(i_\lambda)} n_{i_\lambda, u} c_u^{(\lambda-1)} + \sum_{p \in I} \delta_{i_\lambda p}.
    \]
    The net change in the number of chips at vertex $i_\lambda$ is the difference $\Delta c_{i_\lambda} = c_{i_\lambda}^{(\lambda)} - c_{i_\lambda}^{(\lambda-1)}$. Using the adjacency coefficients $n_{vu} = -A_{uv} = -\langle \alpha_u, \alpha_v^\vee \rangle$ for $u \neq v$ and the fact that $\langle \alpha_v, \alpha_v^\vee \rangle = 2$, we can rewrite this net change algebraically:
    \begin{align*}
        \Delta c_{i_\lambda} 
        &= -2c_{i_\lambda}^{(\lambda-1)} - \sum_{u \neq i_\lambda} \langle \alpha_u, \alpha_{i_\lambda}^\vee \rangle c_u^{(\lambda-1)} + \sum_{p \in I} \delta_{i_\lambda p} \\
        &= -\left\langle \sum_{u=1}^n c_u^{(\lambda-1)} \alpha_u, \alpha_{i_\lambda}^\vee \right\rangle - \langle \beta, \alpha_{i_\lambda}^\vee \rangle \\
        &= -\langle \mathbf{c}^{(\lambda-1)} + \beta, \alpha_{i_\lambda}^\vee \rangle.
    \end{align*}
    From our algebraic model, we know that $v_{\lambda-1} = \mathbf{c}^{(\lambda-1)} + \beta$, where $v_{\lambda-1} = w_{\lambda-1}(\beta)$ and $w_{\lambda-1} = s_{i_{\lambda-1}} \cdots s_{i_1}$ (with $w_0 = \mathrm{id}$). Therefore, the net change in chips is exactly the integer validity condition $K_\lambda$:
    \[
    \Delta c_{i_\lambda} = -\langle v_{\lambda-1}, \alpha_{i_\lambda}^\vee \rangle = K_\lambda.
    \]
    Since the values at all other vertices $u \neq i_\lambda$ remain invariant during step $\lambda$, the total number of chips on the whole diagram increases by exactly $K_\lambda$. Starting from the empty configuration $|\mathbf{c}^{(0)}| = 0$, the final total chip count is the sum over all steps:
    \[
    |\mathbf{c}_{final}^{(I)}| = \sum_{\lambda=1}^m K_\lambda.
    \]
    In the proof of Theorem \ref{thm:bijection}, we established that $K_\lambda = \mathrm{ht}_I\bigl(w_{\lambda-1}^{-1}(\alpha_{i_\lambda}^\vee)\bigr)$. 
    Furthermore, as $(i_1, \dots, i_m)$ is a reduced expression for $w_0^J$, the set of coroots $\{ w_{\lambda-1}^{-1}(\alpha_{i_\lambda}^\vee) \}_{\lambda=1}^m$ is exactly the inversion set of $w_0^J$, which is well‑known to coincide with $(\Phi^\vee)^+ \setminus (\Phi_J^\vee)^+$.
    
    Substituting this into the sum, we obtain:
    \[
    |\mathbf{c}_{final}^{(I)}| = \sum_{\lambda=1}^m \mathrm{ht}_I\bigl(w_{\lambda-1}^{-1}(\alpha_{i_\lambda}^\vee)\bigr) 
    = \sum_{\gamma^\vee \in (\Phi^\vee)^+ \setminus (\Phi_J^\vee)^+} \mathrm{ht}_I(\gamma^\vee),
    \]
    which completes the proof.
\end{proof}

\begin{remark}
    \label{rem:root_counting_special_cases}
    \begin{enumerate}
        \item If $I = \{j\}$ is a single vertex, then $\mathrm{ht}_I(\gamma^\vee)$ is simply the coefficient of $\alpha_j^\vee$ in the simple expansion of $\gamma^\vee$. In this case, Theorem \ref{thm:root_counting_identity} reduces to Theorem~3.14 of \cite{CaviedesCastro2022}, which states that the height of the unique terminal configuration of the Modified Kostant Game on the coroot diagram at vertex $j$ equals the sum of the $j$-th coefficients of all positive roots outside $\Phi_J^+$. Consequently, one recovers the formula $\sum_{\alpha \in \Phi^+} \alpha = \sum_{\alpha_j \in S} h_j \alpha_j$, where $h_j+1$ denotes the height of that terminal configuration.
        \item If $I = S$ (all vertices are modified), then $J = \emptyset$ and $(\Phi_J^\vee)^+ = \emptyset$. The $I$-height of a coroot is just its usual height, and Theorem \ref{thm:root_counting_identity} says that the total number of chips in the final configuration equals the sum of the heights of all positive coroots. In simply‑laced types this number is $|\Phi^+|$ times the height of the highest root.
    \end{enumerate}
\end{remark}

\subsection{Regularity of Reduced Word Languages}

In geometric group theory, a fundamental question is whether the set of reduced words of a Coxeter group forms a regular language. For the minimal length representatives $W^J$, the Modified Kostant Game provides a direct, constructive proof of regularity by naturally defining a Deterministic Finite Automaton (DFA).

Since the number of valid configurations in the game is finite (bounded by the size of $W^J$), we can define a DFA where:
\begin{itemize}
	\item The \textbf{states} are the valid game configurations $c$.
	\item The \textbf{alphabet} is the set of simple reflections $S$.
	\item The \textbf{transitions} are given by the game moves: reading a symbol $s_v$ from state $c$ transitions to $c'$ if vertex $v$ is sad in $c$. If $v$ is happy, the transition goes to a \textit{trap state}.
	\item All valid configurations are \textbf{accepting states}.
\end{itemize}

Because the game strictly builds reduced expressions of $W^J$ without backtracking, the language accepted by this automaton is precisely the language of reduced words in $W^J$.

\begin{theorem}[Regularity of the Coset Representative Language]
	Let $W$ be a finite Weyl group with a set of generators $S = \{s_1, \dots, s_n\}$, and let $W_J$ be a standard parabolic subgroup for some $J \subseteq S$. Let $W^J$ denote the set of minimal length representatives for the left cosets in $W/W_J$. The language $\mathcal{L}(W^J)$, consisting of all reduced words over the alphabet $S$ representing the elements of $W^J$, is a regular language.
\end{theorem}

\begin{proof}
	The proof is constructive. We will prove the theorem by constructing a Deterministic Finite Automaton (DFA), denoted as $M$, which recognizes precisely the language $\mathcal{L}(W^J)$. The construction relies on the bijection between the valid move sequences of the Modified Kostant Game and the words in $\mathcal{L}(W^J)$, established in Theorem \ref{thm:bijection}.
	
	We define the DFA $M = (Q, \Sigma, \delta, q_0, F)$ as follows:
	
	\begin{itemize}
		\item \textbf{States ($Q$):} The set of states is the set of all \textbf{configurations} $c$ reachable in the Modified Kostant Game (associated with $W/W_J$), plus a non-accepting ``trap'' (or sink) state, $q_{\text{trap}}$.
		$$ Q = \{ c_w \mid w \text{ is a valid sequence of moves} \} \cup \{q_{\text{trap}}\} $$
		Where $c_w$ is the configuration reached after the sequence $w$. Since $W^J$ is finite, the number of valid sequences is finite, and thus $Q$ is a finite set.
		
		\item \textbf{Alphabet ($\Sigma$):} The alphabet of the automaton is the set of simple generators of $W$.
		$$ \Sigma = S = \{s_1, \dots, s_n\} $$
		
		\item \textbf{Initial State ($q_0$):} The initial state is the game configuration before any move, $c_{\varepsilon}$, corresponding to the empty word $\varepsilon$.
		
		\item \textbf{Transition Function ($\delta$):} The transition function $\delta: Q \times \Sigma \to Q$ is defined directly from the game rules. For a valid configuration $c \in Q \setminus \{q_{\text{trap}}\}$ and a generator $s_i \in \Sigma$:
		$$ \delta(c, s_i) =
		\begin{cases}
			c' & \text{if vertex } i \text{ is sad in } c \text{ (valid move)}, \\
			q_{\text{trap}} & \text{if vertex } i \text{ is not sad in } c \text{ (invalid move)}.
		\end{cases}
		$$
		Here, $c'$ denotes the new configuration obtained after firing vertex $i$. For the trap state, the transition is absorbing: $\delta(q_{\text{trap}}, s_i) = q_{\text{trap}}$ for all $s_i \in \Sigma$.
		
		\item \textbf{Accepting States ($F$):} A configuration state $c_w$ is an accepting state if the word $w$ leading to it is a reduced word for an element in $W^J$. By Theorem \ref{thm:bijection}, \emph{every} sequence of valid moves $w$ produces an element of $W^J$. Therefore, any reachable valid configuration is an accepting state:
		$$F = Q \setminus \{q_{\text{trap}}\}.$$
	\end{itemize}
	
	Now, we must demonstrate that the language accepted by this automaton, $L(M)$, is exactly $\mathcal{L}(W^J)$.
	
	($\supseteq$) Let $w \in \mathcal{L}(W^J)$. By definition, $w$ is a reduced word for an element in $W^J$. By Theorem \ref{thm:bijection}, $w$ must correspond to a sequence of valid moves in the Modified Kostant Game. By the definition of $\delta$, processing the word $w$ from $q_0 = c_{\varepsilon}$ never enters $q_{\text{trap}}$ and ends in $c_w \in F$. Hence $M$ accepts $w$, so $\mathcal{L}(W^J) \subseteq L(M)$.
	
	($\subseteq$) Let $w$ be accepted by $M$. Then the sequence of transitions for $w$ starting from $q_0$ avoids $q_{\text{trap}}$ and ends in some $c_w \in F$. Thus each step was a valid move, and by Theorem \ref{thm:bijection}, $w$ is a reduced word for an element of $W^J$. Consequently, $w \in \mathcal{L}(W^J)$, so $L(M) \subseteq \mathcal{L}(W^J)$.
	
	Both inclusions give $L(M) = \mathcal{L}(W^J)$. Since we have explicitly constructed a DFA recognizing $\mathcal{L}(W^J)$, the language is regular by Kleene's theorem (see, e.g., \cite{HopcroftUllman1979}).
\end{proof}

As an immediate consequence, taking $J = \emptyset$ (i.e., $I = S$) we recover the well‑known regularity of the full reduced word language of a finite Weyl group.

\begin{corollary}[Recovery of the classical result]
	If all vertices are modified ($I=S$, $J=\emptyset$), then $W^J = W$. The DFA constructed above recognizes the language $\mathcal{L}(W)$ of all reduced words of $W$, recovering the result of Björner and Brenti \cite[\S 4.8]{BjornerBrenti2005} via a purely combinatorial construction.
\end{corollary}

The previous proof builds the automaton $M$ with game configurations as states. By Theorem \ref{thm:bijection} there is a bijection between valid sequences of moves and elements of $W^J$. This induces an isomorphic automaton $M'$ whose states are the elements of $W^J$ themselves. Recall that the sequence of moves $(i_1, \dots, i_t)$ corresponds to the Weyl group element $w = s_{i_t} \cdots s_{i_1}$. When we read the word $s_{i_1} s_{i_2} \cdots s_{i_t}$ from left to right, we build the element by successive left multiplications:
\[
e \xrightarrow{s_{i_1}} s_{i_1} \xrightarrow{s_{i_2}} s_{i_2}s_{i_1} \xrightarrow{s_{i_3}} \cdots \xrightarrow{s_{i_t}} s_{i_t}\cdots s_{i_1} = w .
\]
Therefore, in the automaton $M'$ the transition labelled $s_i$ from a state $w$ goes to $s_i w$, and this transition is valid exactly when $\ell(s_i w) > \ell(w)$ (which guarantees that the word remains reduced) \emph{and} the new state $s_i w$ still belongs to $W^J$. The initial state is $e$, and all states are accepting.

\begin{example}[DFA for $\mathcal{L}(W^J)$ in type $A_2$ with $J = \{s_1\}$]
	Consider $W = W(A_2)$, generated by $S = \{s_1, s_2\}$ with the relation $s_1s_2s_1 = s_2s_1s_2$. Let $J = \{s_1\}$, so the parabolic subgroup is $W_J = \{e, s_1\}$. The active set in the game is $I = S \setminus J = \{s_2\}$.
	
	The set of minimal length left coset representatives is
	\[
	W^J = \{ e,\; s_2,\; s_1s_2 \},
	\]
	and the corresponding reduced words are $\varepsilon$, $s_2$, and $s_2 s_1$ (because reading $s_2$ then $s_1$ from left to right produces the element $s_1 s_2$).
	
	We construct the DFA $M' = (Q, \Sigma, \delta, q_0, F)$ with states $Q = W^J \cup \{q_{\text{trap}}\}$, alphabet $\Sigma = \{s_1, s_2\}$, initial state $q_0 = e$, and accepting states $F = W^J$. The transition function $\delta$ is:
	\[
	\begin{aligned}
	\delta(e, s_1) &= q_{\text{trap}}, \quad & \delta(e, s_2) &= s_2, \\
	\delta(s_2, s_1) &= s_1s_2, \quad & \delta(s_2, s_2) &= q_{\text{trap}}, \\
	\delta(s_1s_2, s_1) &= q_{\text{trap}}, \quad & \delta(s_1s_2, s_2) &= q_{\text{trap}},
	\end{aligned}
	\]
	and $\delta(q_{\text{trap}}, s_i) = q_{\text{trap}}$ for $i=1,2$.
	
	The automaton accepts exactly $\varepsilon$, $s_2$, and $s_2 s_1$, i.e., the language $\mathcal{L}(W^J)$.
	
	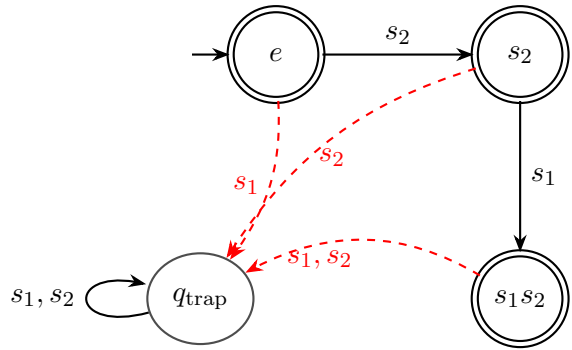
\begin{figure}[H]
		\centering
		\begin{tikzpicture}[
			>=Stealth,
			node distance=2cm and 2cm,
			auto,
			state/.style={circle, draw, thick, minimum size=1.2cm},
			accepting state/.style={state, double, double distance=1.5pt},
			initial text={},
			every edge/.style={draw, ->, thick}
			]
			\node[accepting state, initial] (e) {$e$};
			\node[accepting state] (s2) [right=of e] {$s_2$};
			\node[accepting state] (s1s2) [below=of s2] {$s_1s_2$};
			\node[state, ellipse, draw=black!70] (trap) [below=of e, xshift=-1cm] {$q_{\text{trap}}$};
			
			\path
			(e) edge node[above] {$s_2$} (s2)
			(s2) edge node[right] {$s_1$} (s1s2)
			(e) edge[dashed, red, thick, bend left=20] node[left] {$s_1$} (trap)
			(s2) edge[dashed, red, thick, bend right=20] node[below] {$s_2$} (trap)
			(s1s2) edge[dashed, red, thick, bend right=30] node[below left] {$s_1,s_2$} (trap);
			
			\path (trap) edge [loop left] node {$s_1,s_2$} (trap);
			
		\end{tikzpicture}
		\caption{DFA for $\mathcal{L}(W^{\{s_1\}}) = \{\varepsilon,\; s_2,\; s_2 s_1\}$.}
		\label{fig:dfa_a2}
	\end{figure}
\end{example}

\subsection{Construction of Standard Young Tableaux}

For the symmetric group $S_n$ (the Weyl group of type $A_{n-1}$), the parabolic quotient $S_n / (S_k \times S_{n-k})$ plays a central role in the geometry of Grassmannians: its minimal length representatives, called \textbf{Grassmannian permutations}, parameterize the Schubert cells of the Grassmannian $G(k,n)$ \cite{Winkel1996}. 

A Grassmannian permutation is a permutation $w \in S_n$ whose essential descents are all $\le k$; equivalently, it is the unique element of minimal length in its left coset of $S_n/(S_k \times S_{n-k})$. Each such permutation determines a Young diagram $\lambda(w)$ fitting inside a $k \times (n-k)$ rectangle, and there is a well‑known bijection between the reduced words of $w$ and the Standard Young Tableaux (SYT) of shape $\lambda(w)$, due to Winkel \cite{Winkel1996}.

The Modified Kostant Game with a single source at vertex $k$ provides a visual, dynamical mechanism for this bijection. As the game is played, each firing sequence builds a Standard Young Tableau step by step.

From Theorem \ref{thm:bijection}, the valid move sequences correspond bijectively to reduced words of elements in $W^J$ (with $J = S \setminus \{k\}$). Hence, for a given Grassmannian permutation $w$, every valid game sequence that produces $w$ (and therefore ends in the same final configuration) will construct one of the SYT of shape $\lambda(w)$. Different sequences producing the same $w$ give different SYT, and all SYT of shape $\lambda(w)$ arise in this way.

The filling mechanism is completely determined by the vertex that fires at each step. Recall that the Dynkin diagram of $A_{n-1}$ has vertices labelled $1,2,\dots,n-1$, with the modified vertex being $k$. A Young diagram cell is identified by its row $r$ and column $c$ (rows count from top to bottom, columns from left to right). The \textbf{content} of a cell $(r,c)$ is defined as $c - r$.

The rule to place the next number $j$ in the tableau after the $j$-th move (which fires vertex $i_j$) is:
\begin{equation}\label{eq:content_rule}
    i_j - k = c - r.
\end{equation}
This means that the difference between the vertex fired and the modified vertex determines the diagonal (the content) of the cell that will receive the number $j$. Concretely:
\begin{itemize}
    \item If $i_j = k$ (the modified vertex fires), then $c - r = 0$, so the number is placed on the \textbf{main diagonal} (cells with $c = r$).
    \item If $i_j < k$ (a vertex to the left of $k$ fires), then $c - r < 0$, so $r > c$ and the number goes \textbf{below} the main diagonal (in a lower row).
    \item If $i_j > k$ (a vertex to the right of $k$ fires), then $c - r > 0$, so $c > r$ and the number goes \textbf{above} the main diagonal (in a column to the right).
\end{itemize}

At step $j$, among all empty cells in the current Young shape that satisfy \eqref{eq:content_rule}, the number $j$ is placed in the unique cell that keeps the partial filling a standard Young tableau (i.e., rows strictly increasing and columns strictly increasing). By the properties of Grassmannian permutations, this cell is exactly the addable corner of the shape $\lambda(w_{j-1})$ determined by the left multiplication by $s_{i_j}$.

\begin{example}[The case $S_4 / (S_2 \times S_2)$]
	Consider $n=4$ and $k=2$, which corresponds to the game on $A_3$ with modified vertex $2$. Take the Grassmannian permutation $w$ with reduced word $s_2 s_1 s_3 s_2$. Its associated Young shape is $\lambda = (2,2)$, which fits in the $2 \times 2$ rectangle.
	
	This $w$ admits exactly two reduced words: $s_2 s_1 s_3 s_2$ and $s_2 s_3 s_1 s_2$, corresponding to the two possible orders of the commuting reflections $s_1$ and $s_3$. By Theorem \ref{thm:bijection}, each reduced word yields a valid game sequence, and both lead to the same final configuration $c = (1,2,1)$.
	
	Figure \ref{fig:conjetura_visual_1} shows the two games and the step‑by‑step filling of the two Standard Young Tableaux of shape $(2,2)$:
	\[
	\begin{ytableau}
    1 & 3 \\
    2 & 4
    \end{ytableau}
    \qquad
    \begin{ytableau}
    1 & 2 \\
    3 & 4
    \end{ytableau}
    \]
	The filling rule $i_j - 2 = c - r$ assigns the numbers exactly to the diagonals indicated by the vertex fired: fire $2$ (content $0$) $\to$ main diagonal, fire $1$ (content $-1$) $\to$ one step below diagonal, fire $3$ (content $+1$) $\to$ one step above diagonal. The order in which $s_1$ and $s_3$ are fired distinguishes the two tableaux.
\end{example}

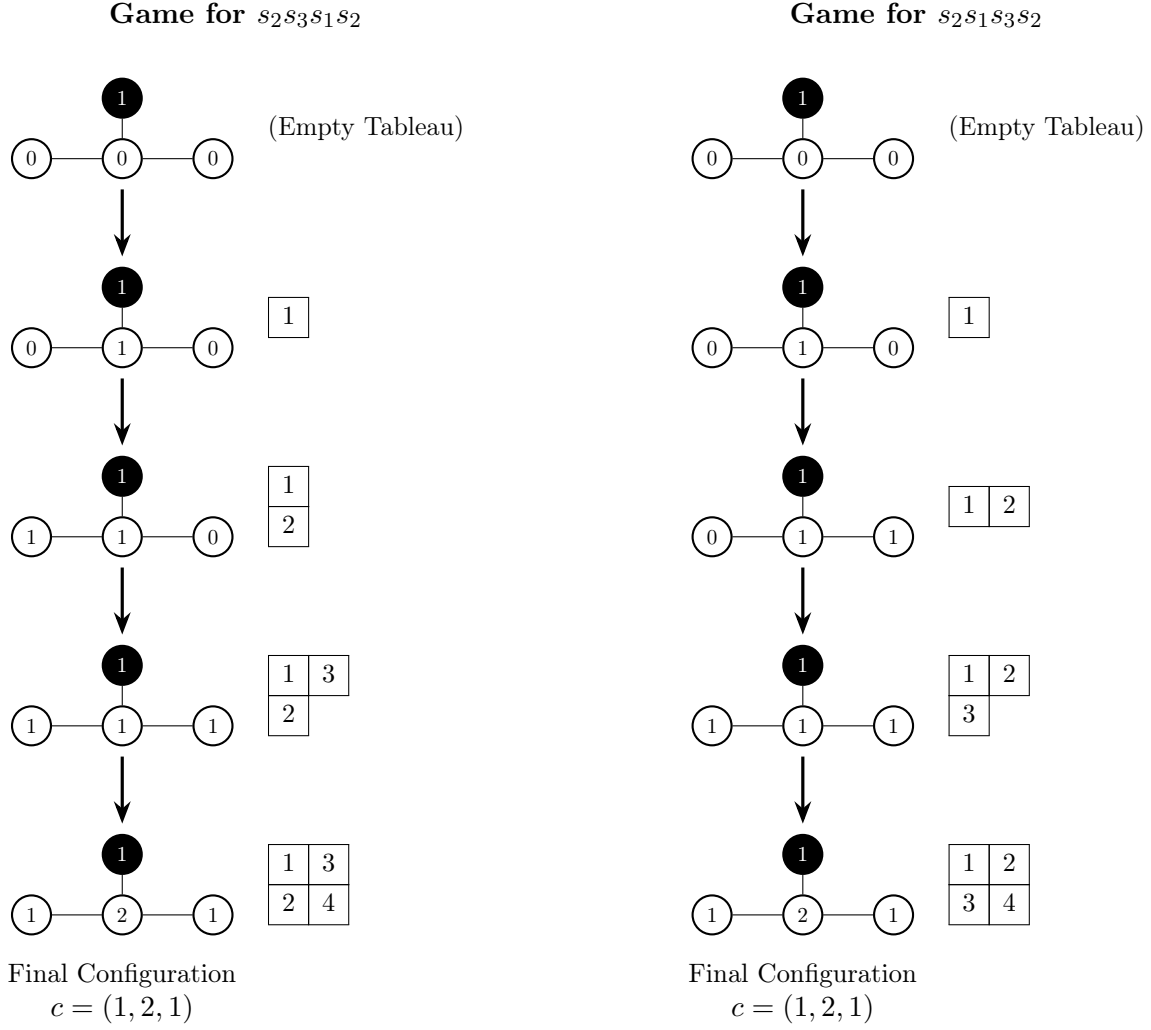
\begin{figure}[H]
	\centering
	\begin{tikzpicture}[
		vtx/.style={circle, draw, fill=white, thick, minimum size=18pt, font=\small},
		src/.style={vtx, fill=black, text=white},
		arrow/.style={-Stealth, very thick},
		game/.style={scale=0.8, transform shape},
		tableau/.style={scale=0.9, transform shape} 
		]
		
		\node (title1) at (1.5,0) {\bfseries Game for $s_2s_3s_1s_2$};
    \node (title2) at (10.5,0) {\bfseries Game for $s_2s_1s_3s_2$};
		
		\node (g1_0) at (0, -1.5) {
			\begin{tikzpicture}[game]
				\node[vtx] (v1) at (0,0) {0}; \node[vtx] (v2) at (1.5,0) {0}; \node[vtx] (v3) at (3,0) {0};
				\node[src] (s2) at (1.5,1) {1}; \draw (v1)--(v2)--(v3); \draw (s2)--(v2);
			\end{tikzpicture}
		};
		\node[tableau, right=0.2cm of g1_0] (t1_0) {(Empty Tableau)};
		
		\node (g1_1) at (0, -4) {
			\begin{tikzpicture}[game]
				\node[vtx] (v1) at (0,0) {0}; \node[vtx] (v2) at (1.5,0) {1}; \node[vtx] (v3) at (3,0) {0};
				\node[src] (s2) at (1.5,1) {1}; \draw (v1)--(v2)--(v3); \draw (s2)--(v2);
			\end{tikzpicture}
		};
		\node[tableau, right=0.2cm of g1_1] (t1_1) {
			\begin{ytableau} 1 \end{ytableau}
		};
		
		\node (g1_2) at (0, -6.5) {
			\begin{tikzpicture}[game]
				\node[vtx] (v1) at (0,0) {1}; \node[vtx] (v2) at (1.5,0) {1}; \node[vtx] (v3) at (3,0) {0};
				\node[src] (s2) at (1.5,1) {1}; \draw (v1)--(v2)--(v3); \draw (s2)--(v2);
			\end{tikzpicture}
		};
		\node[tableau, right=0.2cm of g1_2] (t1_2) {
			\begin{ytableau} 1 \\ 2 \end{ytableau}
		};
		
		\node (g1_3) at (0, -9) {
			\begin{tikzpicture}[game]
				\node[vtx] (v1) at (0,0) {1}; \node[vtx] (v2) at (1.5,0) {1}; \node[vtx] (v3) at (3,0) {1};
				\node[src] (s2) at (1.5,1) {1}; \draw (v1)--(v2)--(v3); \draw (s2)--(v2);
			\end{tikzpicture}
		};
		\node[tableau, right=0.2cm of g1_3] (t1_3) {
			\begin{ytableau} 1 & 3 \\ 2 \end{ytableau}
		};
		
		\node (g1_4) at (0, -11.5) {
			\begin{tikzpicture}[game]
				\node[vtx] (v1) at (0,0) {1}; \node[vtx] (v2) at (1.5,0) {2}; \node[vtx] (v3) at (3,0) {1};
				\node[src] (s2) at (1.5,1) {1}; \draw (v1)--(v2)--(v3); \draw (s2)--(v2);
			\end{tikzpicture}
		};
		\node[tableau, right=0.2cm of g1_4] (t1_4) {
			\begin{ytableau} 1 & 3 \\ 2 & 4 \end{ytableau}
		};
		\node[below=0.1cm of g1_4, align=center] {\small Final Configuration \\ $c=(1,2,1)$};
		
		\node (g2_0) at (9, -1.5) {
			\begin{tikzpicture}[game]
				\node[vtx] (v1) at (0,0) {0}; \node[vtx] (v2) at (1.5,0) {0}; \node[vtx] (v3) at (3,0) {0};
				\node[src] (s2) at (1.5,1) {1}; \draw (v1)--(v2)--(v3); \draw (s2)--(v2);
			\end{tikzpicture}
		};
		\node[tableau, right=0.2cm of g2_0] (t2_0) {(Empty Tableau)};
		
		\node (g2_1) at (9, -4) {
			\begin{tikzpicture}[game]
				\node[vtx] (v1) at (0,0) {0}; \node[vtx] (v2) at (1.5,0) {1}; \node[vtx] (v3) at (3,0) {0};
				\node[src] (s2) at (1.5,1) {1}; \draw (v1)--(v2)--(v3); \draw (s2)--(v2);
			\end{tikzpicture}
		};
		\node[tableau, right=0.2cm of g2_1] (t2_1) {
			\begin{ytableau} 1 \end{ytableau}
		};
		
		\node (g2_2) at (9, -6.5) {
			\begin{tikzpicture}[game]
				\node[vtx] (v1) at (0,0) {0}; \node[vtx] (v2) at (1.5,0) {1}; \node[vtx] (v3) at (3,0) {1};
				\node[src] (s2) at (1.5,1) {1}; \draw (v1)--(v2)--(v3); \draw (s2)--(v2);
			\end{tikzpicture}
		};
		\node[tableau, right=0.2cm of g2_2] (t2_2) {
			\begin{ytableau} 1 & 2 \end{ytableau}
		};
		
		\node (g2_3) at (9, -9) {
			\begin{tikzpicture}[game]
				\node[vtx] (v1) at (0,0) {1}; \node[vtx] (v2) at (1.5,0) {1}; \node[vtx] (v3) at (3,0) {1};
				\node[src] (s2) at (1.5,1) {1}; \draw (v1)--(v2)--(v3); \draw (s2)--(v2);
			\end{tikzpicture}
		};
		\node[tableau, right=0.2cm of g2_3] (t2_3) {
			\begin{ytableau} 1 & 2 \\ 3 \end{ytableau}
		};
		
		\node (g2_4) at (9, -11.5) {
			\begin{tikzpicture}[game]
				\node[vtx] (v1) at (0,0) {1}; \node[vtx] (v2) at (1.5,0) {2}; \node[vtx] (v3) at (3,0) {1};
				\node[src] (s2) at (1.5,1) {1}; \draw (v1)--(v2)--(v3); \draw (s2)--(v2);
			\end{tikzpicture}
		};
		\node[tableau, right=0.2cm of g2_4] (t2_4) {
			\begin{ytableau} 1 & 2 \\ 3 & 4 \end{ytableau}
		};
		\node[below=0.1cm of g2_4, align=center] {\small Final Configuration \\ $c=(1,2,1)$};
		
		\draw[arrow] (g1_0) -- (g1_1); \draw[arrow] (g1_1) -- (g1_2); \draw[arrow] (g1_2) -- (g1_3); \draw[arrow] (g1_3) -- (g1_4);
		\draw[arrow] (g2_0) -- (g2_1); \draw[arrow] (g2_1) -- (g2_2); \draw[arrow] (g2_2) -- (g2_3); \draw[arrow] (g2_3) -- (g2_4);
		
	\end{tikzpicture}
	\caption{Two valid games on $A_3$ modified at $s_2$. Both correspond to different reduced expressions of the element $w=s_2s_1s_3s_2$ and construct the two SYTs of shape $\lambda=(2,2)$.}
	\label{fig:conjetura_visual_1}
\end{figure}

\begin{theorem}
\label{thm:syt_construction}
Let $w \in W^J$ be a Grassmannian permutation for $J = S \setminus \{k\}$, and let $\lambda = \lambda(w)$ be its associated shape. The step-by-step filling mechanism that assigns to each sequence of valid moves $(i_1, i_2, \dots, i_m)$ a placement of the integer $j$ in the unique addable cell $(r,c)$ satisfying $c - r = i_j - k$, constitutes a well-defined bijection onto the set $\text{SYT}(\lambda)$ of Standard Young Tableaux of shape $\lambda$.
\end{theorem}

\begin{proof}
The proof proceeds by induction on the length $m = \ell(w)$.

If $m=0$, $w$ is the identity, $\lambda(w)$ is the empty shape, and the empty sequence maps trivially to the empty tableau. If $m=1$, the only valid move must be on the modified vertex, so $i_1 = k$. The shape $\lambda(s_k)$ is a single box $\square = (1,1)$. The content condition yields $c - r = 1 - 1 = 0 = i_1 - k$, so the entry $1$ is placed in $(1,1)$, forming a valid SYT.

Assume the construction is a well-defined bijection for all elements in $W^J$ of length $m-1$. Let $w \in W^J$ with $\ell(w) = m$, and let $(i_1, \dots, i_{m-1}, i_m)$ be a valid game sequence producing $w$.

Denote by $w' = s_{i_{m-1}} \cdots s_{i_1}$ the element obtained after $m-1$ moves. Then $w = s_{i_m} w'$ and $\ell(w) = \ell(w') + 1$, so $w$ covers $w'$ in the left weak Bruhat order. For Grassmannian permutations, this covering relation corresponds exactly to adding one addable corner to the Young diagram; hence $\lambda(w) = \lambda(w') \cup \{(r^*, c^*)\}$ for a uniquely determined cell. The content of that cell is forced by the vertex that fires: by the properties of the Winkel bijection \cite{Winkel1996}, left multiplication by $s_{i_m}$ adds a box whose content is precisely $i_m - k$. Therefore the cell assigned to the integer $m$ by our rule is exactly $(r^*, c^*)$.

Since $m$ is the maximum value placed so far and $(r^*, c^*)$ was an addable corner of $\lambda(w')$, the resulting filling remains a valid SYT.

Now to prove the bijectivity, given $T \in \text{SYT}(\lambda)$, the box containing the maximum entry $m$ is an inner corner of $\lambda$. Removing that box gives an SYT $T'$ of a shape $\mu \subset \lambda$ that also fits in the $k \times (n-k)$ rectangle, hence $\mu = \lambda(w')$ for a unique Grassmannian permutation $w' \in W^J$ of length $m-1$. The content $c_m - r_m$ of the removed box determines the last move: $i_m = k + c_m - r_m$, and because the box was an addable corner of $\mu$, we have $\ell(s_{i_m} w') = \ell(w') + 1$, so $i_m$ is a valid move. By the induction hypothesis, $T'$ corresponds to a unique valid game sequence $(i_1, \dots, i_{m-1})$ for $w'$, and appending $i_m$ yields the desired sequence for $T$.

Thus the filling mechanism is a bijection between valid game sequences for $w$ and SYT of shape $\lambda(w)$.
\end{proof}


\subsection{The Mukai Conjecture and the Kostant game}

The Mukai conjecture asserts that if $M$ is a Fano variety of complex dimension $n$, index $k_0$, and Picard number $b$, then
\[
n \geq b(k_0-1),
\]
with equality if and only if
$M \cong (\mathbb{CP}^{k_0-1})^b$
\cite{mukai}.

The conjecture has been established in several important settings. In particular, it was proved by Casagrande for toric Fano varieties \cite{casagrande2006}, by Pasquier for horospherical varieties \cite{pasquier}, and more recently by Reineke for quiver moduli varieties \cite{Reineke2024}.

For coadjoint orbits, the conjecture was verified by the first author together with Milena Pabiniak and Silvia Sabatini in \cite{CaviedesCastro2022}, using the combinatorics of the Kostant game.
More precisely, let $G$ be a compact Lie group, let $G_{\mathbb C}$ denote its complexification, and let $P \subset G_{\mathbb C}$ be a parabolic subgroup containing a fixed Borel subgroup $B$. Denote by $\Phi$ the corresponding root system and by $\Delta$ the set of simple roots determined by $B$. Let $\Delta_P \subset \Delta$ be the subset corresponding to $P$. Then the Picard number of the coadjoint orbit $G_{\mathbb C}/P$ is given by
\[
|\Delta \smallsetminus \Delta_P|,
\]
its index is
\[
\gcd_{\beta\in \Delta \smallsetminus \Delta_P}
\Biggl(
\sum_{\alpha\in \Phi^+\smallsetminus \Phi_{\Delta_P}^+}
\langle \alpha,\beta^\vee\rangle - 1
\Biggr),
\]
and its complex dimension is
\[
|\Phi^+\smallsetminus \Phi_{\Delta_P}^+|.
\]
Here $\Phi^+$ and $\Phi_{\Delta_P}^+$ denote the sets of positive roots generated by $\Delta$ and $\Delta_P$, respectively. Therefore, verifying the Mukai conjecture for $G_{\mathbb C}/P$ is equivalent to proving the inequality
\[
|\Delta \smallsetminus \Delta_P|\cdot
\Bigl(
\gcd_{\beta\in \Delta \smallsetminus \Delta_P}(n_\beta)-1
\Bigr)
\le
|\Phi^+\smallsetminus \Phi_{\Delta_P}^+|,
\]
where
\[
n_\beta=
\sum_{\alpha\in \Phi^+\smallsetminus \Phi_{\Delta_P}^+}
\langle \alpha,\beta^\vee\rangle.
\]
See \cite{CaviedesCastro2022} for further details.

In this section, we revisit this argument and present a self-contained proof of the above inequality using the combinatorics of the modified Kostant game.
We start by recalling two lemmas that will be used throughout the proof and can be found in \cite{Humphreys1972}.

\begin{lemma}\label{root sum lemma}
Let $\beta$ be a simple root. Then
\[
\sum_{\alpha\in \Phi^+}\langle \alpha, \beta^\vee \rangle=2.
\]
\end{lemma}

\begin{lemma}\label{lemma:string}
Let $\alpha$ and $\beta$ be non-proportional roots. Let $r, q\in
\Z_{\geq 0}$ be the largest integers for which $\beta-r\alpha\in \Phi,
\beta+q\alpha\in \Phi$. Then $\beta+i\alpha\in \Phi$ for all $-r\leq i\leq
q$. The set $\{\beta+i\alpha : i\in \Z\}$ is called the
\textbf{$\alpha$-string through $\beta$}.
\end{lemma}

\begin{theorem}
For any subset $\Delta_P \subset \Delta$, the inequality
\begin{equation}\label{strong inequality}
\sum_{\beta\in \Delta \smallsetminus \Delta_P}
\Biggl(
\sum_{\alpha\in \Phi^+\smallsetminus \Phi_{\Delta_P}^+}
\langle \alpha,\beta^\vee\rangle-1
\Biggr)
\leq |\Phi^+\smallsetminus \Phi_{\Delta_P}^+|
\end{equation}
holds. As a consequence,
\[
|\Delta \smallsetminus \Delta_P|\cdot
\Biggl(
\gcd_{\beta\in \Delta \smallsetminus \Delta_P}
\Biggl(
\sum_{\alpha\in \Phi^+\smallsetminus \Phi_{\Delta_P}^+}
\langle \alpha,\beta^\vee\rangle
\Biggr)
-1
\Biggr)
\le
|\Phi^+\smallsetminus \Phi_{\Delta_P}^+|.
\]
\end{theorem}

\begin{proof}
Consider the subset $\tilde{\Phi}_{P}\subset (\Phi^\vee)^+$ consisting of all positive coroots of the form
\[
\beta^\vee+\sum_{\alpha_j\in \Delta_P} n_j\alpha_j^\vee,
\]
where $\beta\in \Delta \smallsetminus \Delta_P$ and $n_j\in\mathbb Z_{\ge 0}$.
Clearly,
\[
\tilde{\Phi}_{P}\subset (\Phi^\vee)^+\smallsetminus (\Phi_{\Delta_P}^\vee)^+.
\]
Thus, it is enough to prove that
\[
\sum_{\beta\in \Delta \smallsetminus \Delta_P}
\Biggl(
\sum_{\alpha\in \Phi^+\smallsetminus \Phi_{\Delta_P}^+}
\langle \alpha,\beta^\vee\rangle-1
\Biggr)
\le |\tilde{\Phi}_{P}|.
\]

Fix $\beta \in \Delta \smallsetminus \Delta_P$. Consider the Dynkin diagram associated with $\Delta_P$, and let $\Gamma_{\Delta_P}$ be the corresponding subgraph. Denote by
\[
\Gamma_{P_j} = \Gamma_{P_j}(\beta)
\]
the connected components of $\Gamma_{\Delta_P}$ that are adjacent to $\beta$. Note that there are at most three such components. Suppose that for a component $\Gamma_{P_j}$ there are $k$ arrows pointing from the simple root in $\Gamma_{P_j}$ adjacent to $\beta$ towards $\beta$.

For each component $\Gamma_{P_j}$, consider a string of coroots
\[
\beta^\vee,\,
\beta^\vee+\gamma_1^j,\,
\ldots,\,
\beta^\vee+\gamma_{l_j}^j
\]
obtained by playing the modified Kostant game on $\Gamma_{P_j}\cup\{\beta\}$ with active set $I=\{\beta\}$ and with the arrow multiplicities dictated by the Dynkin diagram, until termination. Each $\gamma_i^j$ is a positive linear combination of coroots in $(\Phi_{P_j}^\vee)^+$.
If $\beta$ is not adjacent to any vertex of $\Delta_P$, we simply take the string consisting only of $\beta^\vee$.

The mechanics of the Kostant game imply that these strings can be concatenated together. For instance, the following is a string of coroots:
\[
\beta^\vee,\,
\beta^\vee+\gamma_1^1,\,
\ldots,\,
\beta^\vee+\gamma_{l_1}^1,\,
\beta^\vee+\gamma_{l_1}^1+\gamma_1^2,\,
\ldots,\,
\beta^\vee+\sum_j\gamma_{l_j}^j.
\]
By Lemma~\ref{lemma:string}, this string can be completed so that the resulting coroots have consecutive heights. Hence its total length is
\[
1+\sum_j l_j
=1+\sum_j l_j(\beta).
\]
Moreover, every root corresponding to a coroot in this string belongs to $\tilde{\Phi}_{P}$.

It remains to identify this quantity with
\[
\sum_{\alpha\in \Phi^+\smallsetminus \Phi_{\Delta_P}^+}
\langle \alpha,\beta^\vee\rangle-1.
\]
By Lemma~\ref{root sum lemma},
\begin{align*}
\sum_{\alpha\in \Phi^+\smallsetminus \Phi_{\Delta_P}^+}
\langle \alpha,\beta^\vee\rangle-1
&=
1-\sum_{\alpha\in \Phi_{\Delta_P}^+}\langle \alpha,\beta^\vee\rangle \\
&=
1-\sum_j
\sum_{\alpha\in \Phi_{P_j}^+}\langle \alpha,\beta^\vee\rangle.
\end{align*}

For each component $\Gamma_{P_j}$, write
\[
\sum_{\alpha\in \Phi_{P_j}^+}\alpha
=
h_j\alpha_j+\cdots,
\]
where $\alpha_j$ is the simple root in $\Gamma_{P_j}$ adjacent to $\beta$.
Since $\beta^\vee$ is orthogonal to all simple roots in $\Delta_P$ except $\alpha_j$, we obtain
\[
\sum_{\alpha\in \Phi_{P_j}^+}\langle \alpha,\beta^\vee\rangle
=
h_j\langle \alpha_j,\beta^\vee\rangle
=
-kh_j.
\]
On the other hand, by the root counting identity from Theorem~\ref{thm:root_counting_identity} applied with $I=\{\beta\}$, the height of the final coroot of the game on $\Gamma_{P_j}\cup\{\beta\}$ is
\[
\operatorname{ht}(\beta^\vee+\gamma_{l_j}^j)=1+kh_j.
\]
Therefore,
\[
-\sum_{\alpha\in \Phi_{P_j}^+}\langle \alpha,\beta^\vee\rangle
=
\operatorname{ht}(\beta^\vee+\gamma_{l_j}^j)-1.
\]
Substituting this into the previous expression gives
\begin{align*}
\sum_{\alpha\in \Phi^+\smallsetminus \Phi_{\Delta_P}^+}
\langle \alpha,\beta^\vee\rangle-1
&=
1+\sum_j
\bigl(\operatorname{ht}(\beta^\vee+\gamma_{l_j}^j)-1\bigr) \\
&=
\operatorname{ht}\Bigl(\beta^\vee+\sum_j\gamma_{l_j}^j\Bigr) \\
&=
1+\sum_j l_j(\beta).
\end{align*}

Summing over all $\beta\in \Delta \smallsetminus \Delta_P$, we conclude that
\[
\sum_{\beta\in \Delta \smallsetminus \Delta_P}
\Biggl(
\sum_{\alpha\in \Phi^+\smallsetminus \Phi_{\Delta_P}^+}
\langle \alpha,\beta^\vee\rangle-1
\Biggr)
=
\sum_{\beta\in \Delta \smallsetminus \Delta_P}
\bigl(1+\sum_j l_j(\beta)\bigr)
\le |\tilde{\Phi}_{P}|,
\]
which proves the theorem.
\end{proof}

\begin{remark}
The inequality proved in the previous theorem was already established by Pasquier in \cite[Lemma~4.8]{pasquier}. His proof relies on a case-by-case analysis depending on the type of the Dynkin diagram and the corresponding parabolic subgroup, verified through explicit calculations. In contrast, the proof presented here is based on the combinatorics of the Kostant game.
\end{remark}


\section{Computational Implementation}
\label{sec:implementation}

To facilitate the exploration of the game dynamics on arbitrary graphs, verify the theoretical results presented in this paper, and generate the corresponding combinatorial structures, an interactive html application has been developed. 

The software allows users to simulate both the classical and the Modified Kostant Game directly in their web browsers without the need for additional installations, supporting simply-laced and multiply-laced diagrams via directed weighted edges. It provides real-time visualization of the node states and tracks the generated words. The playable interactive simulation is publicly available at the following website:
\begin{center}
    \url{https://sites.google.com/unal.edu.co/juascortescru/kostant-game}
\end{center}

\noindent Furthermore, a dedicated interactive module demonstrating the explicit bijection between the Modified Kostant Game on type $A_n$ and the step-by-step construction of Standard Young Tableaux is available at:
\begin{center}
    \url{https://sites.google.com/unal.edu.co/juascortescru/young-tableux}
\end{center}


\section{Conclusions}
\label{sec:conclusions}

This work demonstrates that the multi-vertex Modified Kostant Game is a robust and highly versatile combinatorial framework. By rigorously establishing the bijection between game configurations and the parabolic quotients of Weyl groups $W/W_J$, we have provided a dynamic, algorithmic tool to explore structures that are traditionally handled through abstract algebra.

The implications of this bijection extend far beyond theoretical representation theory. As shown, the game provides an elegant method for root counting, a constructive proof for the regularity of reduced word languages via finite state automata, and a dynamic engine for generating Standard Young Tableaux. Furthermore, its original application in understanding the factorizations of Hilbert polynomials for the Mukai conjecture highlights its deep connection to algebraic geometry. The accompanying computational implementation ensures that this theoretical framework can be actively utilized and expanded in future research.



\begin{thebibliography}{10}

\bibitem{BjornerBrenti2005}
Björner, A., \& Brenti, F. (2005). \textit{Combinatorics of Coxeter Groups}. Graduate Texts in Mathematics, vol. 231. Springer, Berlin, Heidelberg.

\bibitem{CaviedesCastro2022}
Caviedes Castro, A., Pabiniak, M., \& Sabatini, S. (2023). Generalizing the Mukai Conjecture to the symplectic category and the Kostant game. Pure and Applied Mathematics Quarterly \textbf{19} (2023), no. 4, 1803--1837.

\bibitem{casagrande2006}
C.~Casagrande, \emph{The number of vertices of a Fano polytope}, 
Annales de l'Institut Fourier \textbf{56} (2006), no.~1, 121--130.

\bibitem{chen2017}
Chen, E. (2017). \textit{Topics in Combinatorics; Lecture Notes (Taught by Alexander Postnikov)}. MIT Lecture notes.

\bibitem{HopcroftUllman1979}
J.~E.~Hopcroft \& J.~D.~Ullman,
\emph{Introduction to Automata Theory, Languages, and Computation},
Addison-Wesley, 1979.

\bibitem{Humphreys1972}
Humphreys, J. E. (1972). \textit{Introduction to Lie Algebras and Representation Theory}. Graduate Texts in Mathematics, vol. 9. Springer-Verlag, New York.

\bibitem{mukai}
S. Mukai,
\emph{Problems on characterization of the complex projective space},
Birational Geometry of Algebraic Varieties, Open Problems, Proceedings of the 23rd Symposium of the Taniguchi Foundation at Katata, Japan (1988), 57--60.

\bibitem{pasquier}
\newblock B. Pasquier, \
\newblock \emph{Vari\'et\'es horosph\'eriques de Fano,}
\newblock Bull. Soc. Math. France, {\bf 136}, (2008), no. 2,
195--225\,.

\bibitem{Reineke2024}
M. Reineke,
\textit{The Mukai Conjecture for Fano Quiver Moduli},
Algebras and Representation Theory,
Vol. 27, pp. 1641--1644, 2024. Doi 10.1007/s10468-024-10268-8.

\bibitem{Winkel1996}
R. Winkel,
\textit{A combinatorial bijection between standard Young tableaux and reduced words of Grassmannian permutations},
Séminaire Lotharingien de Combinatoire \textbf{36} (1996), Article B36h.

\end{thebibliography}
\end{document}